\documentclass[12pt,a4paper]{article}
\usepackage{amsfonts,amssymb,amsthm,amsmath}
\usepackage{comment}
\usepackage{color}
\usepackage{epsfig}

\newtheorem{lemma}{Lemma}[section]
\newtheorem{remark}[lemma]{Remark}
\newtheorem{theo}[lemma]{Theorem}

\newcommand{\dis}{\displaystyle}
\newcommand{\epsi}{\varepsilon}

\newcommand{\s}{\sigma}
\newcommand{\bmf}{\mbox{\boldmath$f$}}
\newcommand{\bmb}{\mbox{\boldmath$b$}}
\newcommand{\bmu}{\mbox{\boldmath$u$}}
\newcommand{\bmv}{\mbox{\boldmath$v$}}
\newcommand{\bmw}{\mbox{\boldmath$w$}}

\newcommand{\bmxi}{\mbox{\boldmath$\xi$}}

\newcommand{\la}{\langle}
\newcommand{\ra}{\rangle}
\renewcommand{\tilde}{\widetilde}
\renewcommand{\bar}{\overline}

\newcommand{\grad}{\nabla}

\newcommand{\R}{\mathbb{R}}
\newcommand{\G}{\Gamma}
\newcommand{\pr}{\operatorname{pr}}
\newcommand{\BIGOP}[1]{\mathop{\mathchoice%
{\raise-0.22em\hbox{\huge $#1$}}%
{\raise-0.05em\hbox{\Large $#1$}}{\hbox{\large $#1$}}{#1}}}

\newcommand{\wV}{\widehat{V}}
\newcommand{\wN}{\widehat{N}}
\newcommand{\wH}{\widehat{H}}
\newcommand{\brho}{\boldsymbol{\rho}}
\newcommand{\bomu}{\boldsymbol{\mu}}

\newcommand{\sbmv}{\mbox{\scriptsize{$\bmv$}}}
\makeatletter

\@addtoreset{equation}{section}
\makeatother

\begin{document}

\begin{titlepage}
\title{Mean curvature flow with triple junctions in higher space dimensions}
\author{  Daniel Depner\footnote{Fakult\"at f\"ur Mathematik,  
Universit\"at Regensburg,
93040 Regensburg,
Germany, e-mail: {\sf daniel.depner@mathematik.uni-regensburg.de}}, 
Harald Garcke\footnote{Fakult\"at f\"ur Mathematik,  
Universit\"at Regensburg,
93040 Regensburg,
Germany, e-mail: {\sf harald.garcke@mathematik.uni-regensburg.de}}, and 
Yoshihito Kohsaka\footnote{Muroran Institute of Technology, 
27-1 Mizumoto-cho, Muroran 050-8585, Japan,
e-mail: {\sf kohsaka@mmm.muroran-it.ac.jp}}}
\date{}
\end{titlepage}
\maketitle
\begin{abstract}
We consider mean curvature flow of $n$-dimensional surface clusters. At $(n-1)$-dimensional 
triple junctions an angle condition is required which in the symmetric case reduces to the
well-known 120 degree angle condition. Using a novel parametrization of evolving surface clusters
and a new existence and regularity approach for parabolic equations on surface clusters we show
local well-posedness by a contraction argument in parabolic H\"older spaces. 
\end{abstract}
\noindent{\bf Key words:} Mean curvature flow, triple lines, 
local existence result, parabolic H\"older theory, free boundary problem.

\noindent{\bf AMS-Classification:}  53C44, 35K55, 35R35, 58J35. 







\section{Introduction}\label{sec:intro}

Motion by mean curvature for evolving hypersurfaces in $\mathbb{R}^{n+1}$ is given by
$$V=H \,,$$
where $V$ is the normal velocity and $H$ is the mean curvature of the
evolving surface. Mean curvature flow for closed surfaces is the
$L^2$-gradient flow of the area functional and many results for this
flow have been established over the last 30 years, see e.g. Huisken \cite{Hui84},
Gage and Hamilton \cite{GH86}, Ecker \cite{Eck04}, Giga \cite{G}, Mantegazza \cite{M} and the
references therein. 

Less is known for mean curvature flow of surfaces with boundaries. In
the simplest cases one either prescribes fixed Dirichlet boundary data
or one requires that surfaces meet a given surface with a $90$ degree 
angle. The last situation can be interpreted as the $L^2$-gradient
flow of area taking the side constraint into account that the boundary of
the surface has to lie on a given external surface. 
A setting where the surface is given as a graph was studied by Huisken \cite{Hui89}, who could also 
analyze the long time behaviour in the case where the evolving surface was given 
as the graph over a fixed domain. 
Local well-posedness
for general geometries was shown by Stahl \cite{stahl} who was also able to
formulate a continuation criterion. In addition he showed that
surfaces converge asymptotically to a half sphere before they vanish. 

\begin{figure}[h]
\begin{center} 
 \includegraphics[angle=-90,width=0.21\textwidth]{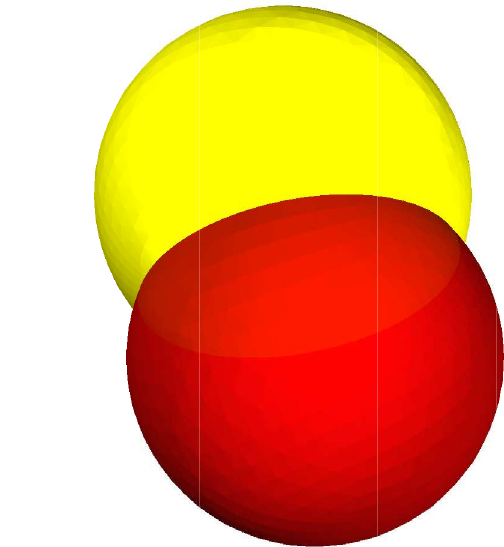} 
 \includegraphics[angle=-90,width=0.14\textwidth]{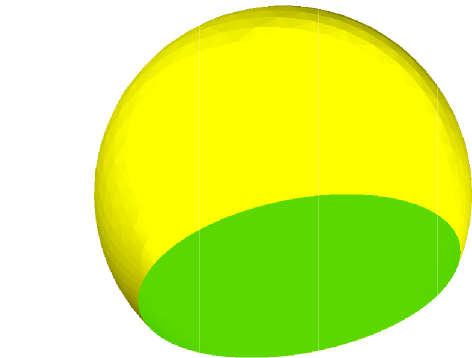}\qquad 
 \includegraphics[angle=-90,width=0.21\textwidth]{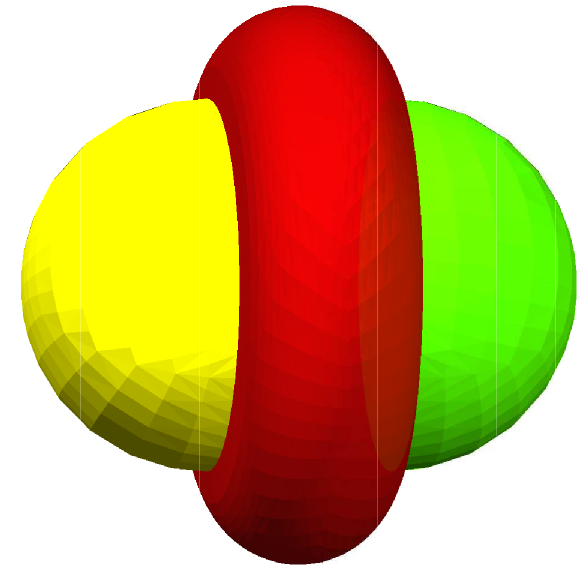} 
 \includegraphics[angle=-90,width=0.14\textwidth]{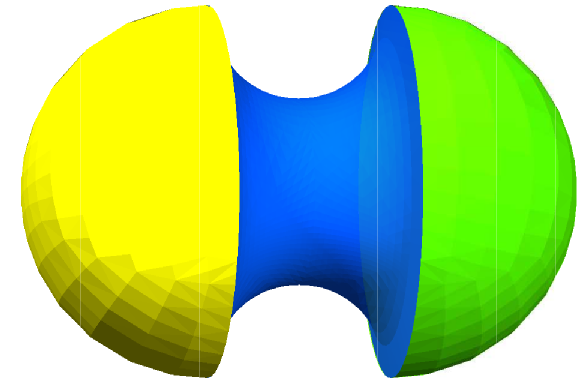} 
\caption{A surface cluster consisting of three hypersurfaces with boundary and one triple line on the left;
         and a surface cluster with four hypersurfaces, where the topology of the individual surfaces is 
         not the same for all on the right; taken from \cite{BGN}.}\label{figure1}
\end{center}
\end{figure}

Much less is known about the gradient flow dynamics for surface
clusters. In this case hypersurfaces $\Gamma^1,\dots,\Gamma^N$ in
$\mathbb{R}^{n+1}$ with boundaries
$\partial\Gamma^1,\dots,\partial\Gamma^N$ meet at $(n-1)$-dimensional
triple junctions, see e.g. Figure \ref{figure1}. Here, 
boundary conditions at the triple junction which can be derived
variationally have to be described. In what follows we briefly discuss
how to derive these boundary conditions. We define the weighted surface
free energy 
$$\mathcal{F}(\Gamma):=\sum^N_{i=1} \int\limits_{\Gamma^i} \gamma^i
d\mathcal{H}^n$$
for a given surface cluster $\Gamma=(\Gamma^1,\dots,\Gamma^N)$ (and
constant surface energy densities $\gamma_i>0$, $i=1,\ldots,N$) and consider a given smooth vector field 
$$\zeta :\mathbb{R}^{n+1}\to \mathbb{R}^{n+1}\,.$$
Then
we can define a variation $\Gamma(\varepsilon)$ of $\Gamma$ in the
direction $\zeta$ via
$$
\Gamma^i(\varepsilon)=\{x+\varepsilon\zeta(x)\mid
x\in\Gamma^i\}\,.$$
A transport theorem now gives 
$$\frac{d}{d\varepsilon} \int\limits_{\Gamma^i(\varepsilon)} 1
d\mathcal{H}^n = -\int\limits_{\Gamma^i(\varepsilon)} V^i H^i d\mathcal{H}^n 
 + \int\limits_{\partial\Gamma^i(\varepsilon)} v^i d\mathcal{H}^{n-1} \,,
$$ 
where $V^i$ is the normal
velocity and $H^i$ is the mean curvature of $\Gamma^i$. In
addition $v^i$ is the outer conormal velocity of the surface, i.e. we have $v^i= \la \zeta, \nu^i \ra$,
where $\nu^i$ is the outer unit conormal of $\partial\Gamma^i$ (for
details we refer to Garcke, Wieland \cite{GW06} and Depner, Garcke \cite{DG11}).

The first variation of $\mathcal{F}$ is now given by
\begin{equation*}
\frac{d}{d\varepsilon}\mathcal{F}(\Gamma(\varepsilon)) = \sum_i
\int_{\Gamma^i(\varepsilon)} \left( -\gamma^i V^i H^i \right) d\mathcal{H}^n+
\sum_i\int_{\partial\Gamma^i(\varepsilon)}\gamma^iv^i d\mathcal{H}^{n-1}
\end{equation*}
and hence a suitably weighted $L^2$-gradient flow is given by
\begin{eqnarray}
V^i = \beta^i H^i &&\mbox{on } \Gamma^i \mbox{ and} \label{MCF1}\\
\sum^3_{i=1}\gamma^i \nu^i=0 &&\mbox{at triple junctions}.
\label{MCF2}
\end{eqnarray}
We remark that the last condition reduces to a $120^\circ$ angle
condition in the case that all $\gamma_i$'s are equal.

Local well-posedness for curves in the plane has been shown by
Bronsard and Rei\-tich~\cite{BR} in a $C^{2+\alpha,1+\frac{\alpha}{2}}$ setting
using parabolic regularity theory and a fixed point argument (for a
typical solution see Figure \ref{curves}). 
Kinderlehrer and Liu \cite{KL} derived global existence of a planar network of grain boundaries 
driven by curvature close to an equilibrium.
Mantegazza, Novaga and Tortorelli \cite{MNT} were able to establish continuation
criteria and Schn\"urer et al. \cite{schnuerer_etal} and Bellettini and Novaga \cite{BN11} considered the asymptotic behaviour of
lens-shaped geometries. We remark that all of these results are restricted to the planar case. 

\begin{figure}[h]
\begin{center}
$\underset{\mbox{equal area}}{\includegraphics[angle=-90,width=0.35\textwidth]{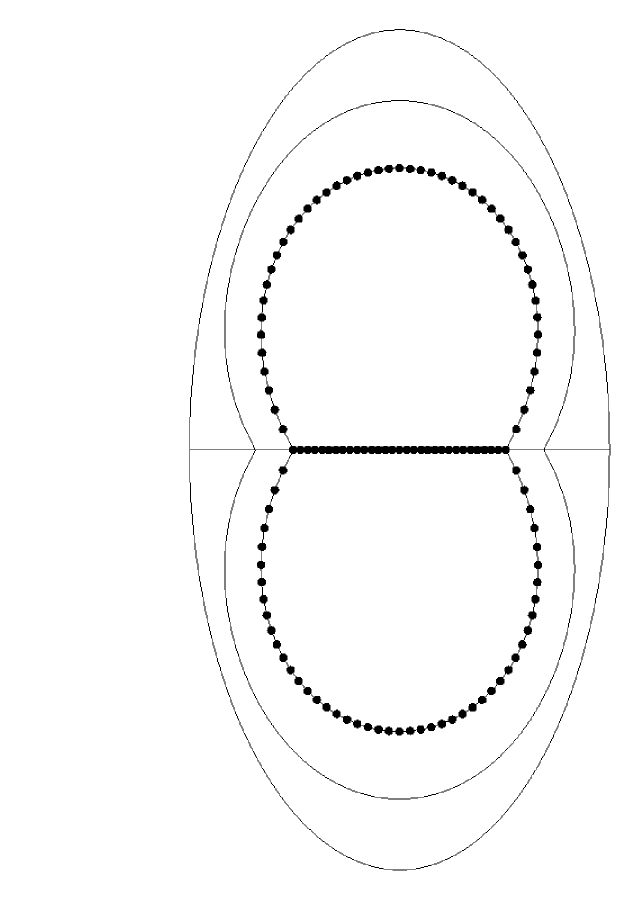}}$
\qquad 
$\underset{\mbox{non-equal area}}{\includegraphics[angle=-90,width=0.35\textwidth]{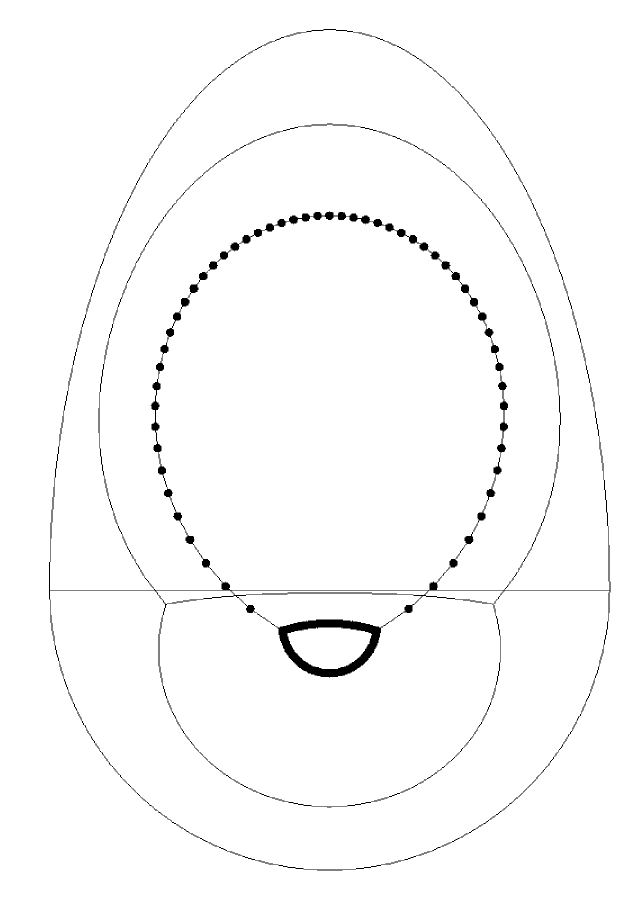}}$
\end{center}
\caption{Mean curvature flow of a double bubble in the plane, see
  \cite{BGN} for results in $\mathbb{R}^3$.}\label{curves}
\end{figure}

The higher dimensional situation is much more involved as the triple
junction now is at least one-dimensional and a tangential degree of
freedom arises at the triple junction. In addition, all mathematical descriptions of the
problem result in formulations which lead to a free boundary 
problem. Only recently, Freire~\cite{freire} was able to show local
well-posedness in the case of graphs. Of course most situations cannot
be represented as graphs. We use a new parametrization of surface
clusters introduced in Depner and Garcke~\cite{DG11} to state the
problem (\ref{MCF1}), (\ref{MCF2}) as a system of non-local, quasilinear
parabolic partial differential equations of second order. The PDEs are
defined on a surface cluster and are non-trivially coupled at the
junctions. To simplify the presentation, we will now stick to the situation of three 
surfaces meeting at one common triple junction. But we remark that 
generalizations of our approach to more general surface clusters are possible
as long as different triple junctions do not meet. Of course this can happen
for soap bubble clusters, see Taylor~\cite{T} and Morgan~\cite{Mor}.
In addition we want to remark that in the situation on the left in Figure \ref{figure1}
it is in principle possible to use one global parametrization for all
three evolving hypersurfaces. In this case we would get a system of PDEs on 
one reference configuration. Due to the topological restrictions this is not possible 
any more in the situation on the right in Figure \ref{figure1}. But since we only use local parametrizations, 
our method works also in this case. 

We hence look for families of evolving hypersurfaces 
$\G^i(t)\subset \mathbb{R}^{n+1}$ ($i=1,2,3$) governed by 
the mean curvature flow, which is weighted by $\beta^i>0$\,($i=1,2,3$). These hypersurfaces 
meet at their boundaries as follows
$$
\partial\G^1(t)=\partial\G^2(t)=\partial\G^3(t)\,(=:\Sigma(t)),
$$
which is an $(n-1)$-dimensional manifold. Also, the angles 
between hypersurfaces are prescribed. More precisely, we consider
\begin{align}
\left\{\begin{array}{l}
V^i=\beta^i H^i
\,\,\ \mbox{on}\,\,\ \G^i(t),\ t\in [0,T]\,\ (i=1,2,3), \\[0.1cm]
\angle(\G^i(t),\G^j(t))=\theta^k
\,\,\ \mbox{on}\,\ \Sigma(t),\ t\in [0,T], \\[0.1cm]
\qquad ((i,j,k)=(1,2,3),\,(2,3,1),\,(3,1,2)), \\[0.1cm]
\G^i(t)|_{t=0}=\G^i_0\,\ (i=1,2,3),
\end{array}\right.
\label{tripwithout}
\end{align}
where $\G^i_0$\,($i=1,2,3$) are given initial hypersurfaces, which meet at 
their boundary, i.e. $\partial\G_0^1=\partial\G_0^2=\partial\G_0^3\,
(=:\Sigma_0)$, and fulfill the angle conditions as above. 
Here, $V^i$ and $H^i$ are the normal velocity and mean curvature of $\G^i(t)$, 
respectively. 

In (\ref{tripwithout}), $\theta^1, \theta^2$ and $\theta^3$ are given 
contact angles with $0 < \theta^i < \pi$, which fulfill 
$\theta^1 + \theta^2 + \theta^3 = 2 \pi$ and Young's law
\begin{equation} \label{younglaw}
 \frac{\sin \theta^1}{\gamma^1} = \frac{\sin \theta^2}{\gamma^2} 
 = \frac{\sin \theta^3}{\gamma^3} \;.
\end{equation}
Let $\nu^i(\cdot,t)$\,($i=1,2,3$) be the outer conormals at $\partial \G^i(t)$. 
Then, introducing the angle conditions as in (\ref{tripwithout}), one can show 
that \eqref{younglaw} is equivalent to
\begin{equation}
\gamma^1\nu^1(\cdot,t)+\gamma^2\nu^2(\cdot,t)+\gamma^3\nu^3(\cdot,t)=0
\,\,\ \mbox{on}\,\,\ \Sigma(t) ,
\label{younglaw2}
\end{equation}
which is the condition \eqref{MCF2} stated above.  
To choose appropriate normals $N^i(\cdot,t)$ of $\G^i(t)$, we observe that 
due to the appearance of a triple junction $\Sigma(t)$ 
the six vectors $N^i(\cdot,t)$, $\nu^i(\cdot,t)$, $i=1,2,3$ on 
$\Sigma(t)$ all lie in a two-dimensional space, namely the orthogonal 
complement $\left( T_\s \Sigma(t) \right)^\perp$ of the triple junction.
In this two-dimensional space we choose an oriented basis and a corresponding 
counterclockwise rotation $R$ around $90$ degree. Then we set 
\begin{align*}
 N^i(\cdot,t):=R \nu^i(\cdot,t) \quad \mbox{ on } \; \Sigma(t)
\end{align*}
and extend these normals by continuity to all of $\Gamma^i(t)$.
Then we can write instead of (\ref{younglaw2})

\begin{equation}
\gamma^1N^1(\cdot,t)+\gamma^2N^2(\cdot,t)+\gamma^3N^3(\cdot,t)=0
\,\,\ \mbox{on}\,\,\ \Sigma(t).
\label{younglaw3}
\end{equation}
In the following the angle conditions at the triple line are written as 
\begin{align} \label{younglaw4}
 &\la N^i(\cdot,t),N^j(\cdot,t)\ra = \cos \theta^k 
\end{align}
on $\Sigma(t)$ for $(i,j,k) = (1,2,3)$, $(2,3,1)$, and $(3,1,2)$. 
Here and hereafter, $\la\,\cdot\,,\,\cdot\,\ra$ means the inner product 
in ${\mathbb R}^{n+1}$. 

We are able to show the following result (for a precise formulation of the
result we refer to Section \ref{sec:locexist}):

\bigskip\noindent
\textbf{Main result.} \\
\textit{Let $(\Gamma^1_0,\Gamma^2_0,\Gamma^3_0)$ be a $C^{2+\alpha}$ surface 
cluster with a $C^{2+\alpha}$ triple junction curve $\gamma$.
We assume the compatibility conditions
\begin{itemize}
\item[-] 
$(\Gamma^1_0,\Gamma^2_0,\Gamma^3_0)$ fulfill the angle conditions,
\item[-] 
$\gamma^1 \beta^1 H^1_0 + \gamma^2 \beta^2 H^2_0 + \gamma^3 \beta^3 H^3_0 = 0$ 
on the triple line $\partial \Gamma^1_0 = \partial \Gamma^2_0 
= \partial \Gamma^3_0$.
\end{itemize}
Then there exists a local
$C^{2+\alpha,1+\frac{\alpha}{2}}$ solution of 
\begin{eqnarray*}
&V^i=H^i\,\,\ +\,\,\ \mbox{angle conditions},
\end{eqnarray*}
with initial data $(\Gamma^1_0,\Gamma^2_0,\Gamma^3_0)$.
}

\bigskip
The idea of the proof is as follows:
First we study the linearized problem around a reference configuration with energy methods (this is
non-trivial as the system is defined on a surface cluster). 
Then we show local $C^{2+\alpha,1+\frac{\alpha}{2}}$-regularity of the solutions to the linearized 
problem. In order to apply classical regularity theory close to the triple junction, we parametrize 
the cluster locally over one fixed reference domain and check the Lopatinskii-Shapiro condition for the 
resulting spatially localized system on the flat reference domain directly and for convenience with an energy argument. 
Finally we use a fixed point argument in $C^{2+\alpha,1+\frac{\alpha}{2}}$ which is non-trivial as the
overall system is non-local. In this context ideas of Baconneau and Lunardi \cite{BL} are useful.

We remark that we do not need the initial surfaces $\Gamma^i_0$ to be of class $C^{3+\alpha}$ as in 
\cite{BL} since we linearize around smooth enough reference hypersurfaces, which are 
close enough to $\Gamma^i_0$ in the $C^{2+\alpha}$-norm.  

We also remark that the overall problem has a structure similar as free boundary problems. 
This is due to the fact that at the triple junction a motion of the surface cluster in 
conormal direction is necessary. When formulating the evolution on a fixed reference 
configuration, we need to take care of the conormal velocity which results in a highly 
nonlinear nonlocal evolution problem similar as in several free boundary problems, 
see e.g. Escher and Simonett \cite{ES97} or Baconneau and Lunardi \cite{BL}.
In our context an additional difficulty arises due to the fact that three surfaces
who all have a conormal velocity meet at the triple junction. The connection to 
free boundary problems is more apparent in the graph case which has been 
considered by Freire \cite{freire}.


\section{PDE formulation}\label{sec:param}

\subsection{Parametrization of surface clusters} \label{subsec:paramsurfclus}

Let us describe $\G^i(t)$ with the help of functions 
$\rho^i : \G^i_\ast \times [0,T] \to \mathbb{R}$ 
as graphs over some fixed compact reference hypersurfaces $\G^i_\ast$\,($i=1,2,3$) 
of class $C^{3+\alpha}$ for some $0<\alpha<1$
with boundary $\partial\Gamma^i_\ast$. These are supposed to have 
a common boundary 
\begin{equation} \label{fixcommbound}
\partial \G^1_\ast = \partial \G^2_\ast = \partial \G^3_\ast\,
(=:\Sigma_\ast)
\end{equation}
and fulfill the angle conditions from \eqref{tripwithout}. As above, 
we introduce notation such that the outer conormals 
$\nu^i_\ast$ at $\partial \G^i_\ast$ fulfill
$$
 \gamma^1 \nu^1_\ast + \gamma^2 \nu^2_\ast + \gamma^3 \nu^3_\ast = 0
 \quad \mbox{on } \; \Sigma_\ast \;,
$$
and the normals $N^i_\ast$ of $\G^i_\ast$ are chosen such that
\begin{align} \label{sum_N}
 \gamma^1 N^1_\ast + \gamma^2 N^2_\ast + \gamma^3 N^3_\ast = 0
 \quad \mbox{on } \; \Sigma_\ast \;.
\end{align}
Note that we do not assume $\G^i_\ast$ to be a stationary solution of 
\eqref{tripwithout}, that is the mean curvature of $\G^i_\ast$ can 
be arbitrary. 

Let $F^i\,:\,\Omega^i \rightarrow \mathbb{R}^{n+1}$ be 
a local parametrization with $F^i(\Omega^i) \subset \G_*^i$ where 
$\Omega^i$ is either an open subset of $\mathbb{R}^n$ or 
$B^+(0) = \{x \in \R^n \,|\, |x| < 1r \,, \; x_n \geq 0\}$
in the case that we parametrize around a boundary point. 
For $\sigma\in\G_*^i$, 
we set $F^{-1}(\sigma)=(x_1(\sigma),\ldots,x_n(\sigma))\in\mathbb{R}^n$. 
Here and hereafter, for simplicity, we use the notation
$$
w(\sigma)=w(x_1,\ldots,x_n) \quad (\sigma\in\G_*^i) \,,
$$
i.e. we omit the parametrization. 
In particular, we set $\partial_l w:=\partial_{x_l}(w \circ F)$. 

To parametrize a hypersurface close to $\Gamma_\ast^i$, 
we define the mapping through
\begin{align} \label{Psidef}
 \Psi^i\,:\, \G^i_\ast \times (-\epsi,\epsi) \times (-\delta,\delta) 
 \rightarrow &\ \mathbb{R}^{n+1} \;,\\  
 (\sigma,w,r) 
 \mapsto &\ \Psi^i(\sigma,w,r) := \sigma + w \, N^i_\ast(\sigma) 
   + r \, \tau^i_\ast(\sigma) \;, \nonumber
\end{align}
where $\tau^i_\ast$ is a tangential vector field on $\G^i_\ast$ with support 
in a neighbourhood of $\partial \G^i_\ast$, which equals the conormal 
$\nu^i_\ast$ at $\partial \G^i_\ast$. The index $i$ has range $1,2,3$. 

For $i=1,2,3$ and functions 
$$
\rho^i\,:\,\G^i_\ast \times [0,T] \rightarrow (-\varepsilon,\varepsilon) \,, 
\quad \mu^i\,:\,\Sigma_\ast \times [0,T] \rightarrow (-\delta,\delta) 
$$
we define the mappings $\Phi^i=\Phi^i_{\rho^i, \mu^i}$ (we often omit 
the subscript $(\rho^i,\mu^i$) for shortness) through
\begin{align*}
 \Phi^i : \G^i_\ast \times [0,T] \rightarrow&\ \R^{n+1} \;, \quad
 \Phi^i(\sigma,t) := \Psi^i(\sigma,\rho^i(\sigma,t),\mu^i(\pr^i(\sigma),t)) \,.
\end{align*}
Herein $\pr^i : \G^i_\ast \to \partial \G^i_\ast$ is 
defined such that $\pr^i(\sigma) \in \partial \Gamma^i_\ast$ is the point on 
$\partial \Gamma^i_\ast$ with shortest distance on $\Gamma^i_\ast$ to $\sigma$.
We remark here that $\pr^i$ is well-defined 
and smooth close to $\partial \Gamma^i_\ast$.
Note that we need this mapping just in a (small) neighbourhood of 
$\partial \G^i_\ast$, because it is used in the product 
$\mu^i(\pr^i(\sigma),t)\, \tau^i_\ast(\sigma)$, where the second term 
is zero outside a (small) neighbourhood of $\partial \G^i_\ast$. 
For small $\varepsilon, \delta > 0$ and fixed $t$ we set
$$
 (\Phi^i)_t\,:\,\G^i_\ast \rightarrow \mathbb{R}^{n+1},\,\,\quad 
 (\Phi^i)_t(\sigma) := \Phi^i(\sigma,t) \,,
$$
and finally we define new hypersurfaces through
\begin{equation} \label{defnewsurf}
 \G_{\rho^i, \mu^i}(t) := \mbox{image} ((\Phi^i)_t) \,.
\end{equation}
We observe that for $\rho^i \equiv 0$ and $\mu^i \equiv 0$ the resulting 
surface is simply $\G_{\rho^i \equiv 0, \mu^i \equiv 0}(t) = \G^i_\ast$ 
for every $t$. 

\begin{remark} \label{rem:localdiff}
 We remark that for $\rho^i \in C^2(\G^i_\ast)$ and $\mu^i \in C^2(\Sigma_\ast)$ small 
 enough in the $C^1(\G^i_\ast)$- resp. $C^1(\Sigma_\ast)$-norm the mapping $(\Phi^i)_t$
 is a local $C^2$-diffeomorphism onto its image. 
 
 In fact, omitting the time variable $t$ and the index $i$ for the moment, choosing a local 
 parametrization and using the above abbreviations we calculate
 \begin{align*}
  \partial_l \Phi &= \partial_l q + \partial_l \rho \, N_\ast + \rho \partial_l N_\ast 
                      + \partial_l (\mu \circ \pr) \, \tau_\ast + (\mu \circ \pr) \partial_l \tau_\ast \,.
 \end{align*}
 A rather lengthy, but elementary calculation for $g_{lk} = \langle \partial_l \Phi , \partial_k \Phi \rangle$ gives
 \begin{align*}
  g_{lk} &= (g_{\ast})_{lk} + P_{lk}(\rho, (\mu \circ \pr), \nabla \rho, \nabla (\mu \circ \pr)) \,,
 \end{align*}
 where $P_{lk}$ is a polynomial with $P_{lk}(0)=0$. With the help of the Leibniz formula for 
 the determinant we can then derive 
 \begin{align*}
  g = \det \left( (g_{lk})_{l,k=1,\ldots,n} \right)
    = g_\ast + P(\rho, (\mu \circ \pr), \nabla \rho, \nabla (\mu \circ \pr)) \,,
 \end{align*}
 where $P$ is a polynomial with $P(0)=0$. Since $g_\ast > 0$ we conclude that for $\rho$
 and $\mu$ small enough in the $C^1$-norms also $g$ is positive. Together with the fact that 
 $(g_{lk})_{l,k=1,\ldots,n}$ is positive semi-definite due to 
 \begin{align} \label{equ:possemidef}
  \sum_{l,k=1}^n \xi_l g_{lk} \xi_k = \sum_{l,k=1}^n \xi_l \langle \partial_l \Phi , \partial_k \Phi \xi_k \rangle
      = \left| \sum_{l=1}^n \xi_l \partial_l \Phi \right|^2 \geq 0 \; \mbox{ for all } \, \xi \in \mathbb{R}^n \,,
 \end{align}
 we conclude the property that $(g_{lk})_{l,k=1,\ldots,n}$ is even positive definite. 
 Hence we obtain a strict inequality in \eqref{equ:possemidef}, whenever $\xi \neq 0$
 and we conclude that $\partial_1 \Phi,\ldots,\partial_n \Phi$
 are linearly independent, which means that the differential $d \Phi(\sigma)$ has full rank. 
 
 Finally with the help of the inverse function theorem we conclude  that $(\Phi^i)_t$ is a local
 diffeomorphism and the image $\G_i(t)$ has metric tensor $(g_{lk})_{l,k=1,\ldots,n}$. 
\end{remark}

In the definition of $\Psi^i$ we allow at the triple junction for a movement in normal and tangential direction, 
and hence there are enough degrees of freedom to 
formulate the condition, that the hypersurfaces $\G_i(t)$ meet in 
one triple junction $\Sigma(t)$ at their boundary, through
\begin{equation} \label{meetcond}
 \Phi^1(\sigma,t) = \Phi^2(\sigma,t) = \Phi^3(\sigma,t) \quad 
 \mbox{for } \; \sigma \in \Sigma_\ast\;,\ t \geq 0 \,.
\end{equation}
We rewrite these equations in the following lemma, which was shown in 
Depner and Garcke \cite{DG11}.

\begin{lemma} \label{lem:equmeetcond}
Equivalent to the equations \eqref{meetcond} are the following conditions
 \begin{align} \label{fourcond}
 \left\{ \begin{array}{cll}
  (i)\,& \gamma^1 \rho^1 + \gamma^2 \rho^2 + \gamma^3 \rho^3 = 0 
       & \mbox{on } \, \Sigma_\ast \,, \\[0.1cm]
  (ii) & \mu^i = \dfrac{1}{s^i} \left( c^j \rho^j - c^k \rho^k \right) 
       & \mbox{on } \, \Sigma_\ast \,. \rule{0cm}{0.5cm}
 \end{array}\right.
 \end{align}
for $(i,j,k) = (1,2,3), (2,3,1)$ and $(3,1,2)$, and where 
$s^i = \sin \theta^i$ and $c^i = \cos \theta^i$. 
\end{lemma}

With the notation $\bomu = (\mu^1,\mu^2,\mu^3)$, 
$\brho = (\rho^1,\rho^2,\rho^3)$ and the matrix 
\begin{align*}
 {\cal T} = 
 \left( \begin{array}{ccc}
         0 & \dfrac{c^2}{s^1} & -\dfrac{c^3}{s^1} \\[0.25cm]
         -\dfrac{c^1}{s^2} & 0 & \dfrac{c^3}{s^2} \\[0.25cm]
         \dfrac{c^1}{s^3} & -\dfrac{c^2}{s^3} & 0
        \end{array}
 \right)
\end{align*}
we can state the linear dependence from $(ii)$ of \eqref{fourcond} as 
\begin{align} \label{equ:lindep}
 \bomu = {\cal T} \brho \quad \; \mbox{ on } \; \Sigma_\ast \,. 
\end{align}

\subsection{The nonlocal, nonlinear parabolic boundary value problem} \label{subsec:fullprob}

From now on, we always assume condition (\ref{meetcond}). 
We introduce the notation $\widehat{N}^i(\sigma,t)$, $\widehat{V}^i(\sigma,t)$ 
and $\widehat{H}^i(\sigma,t)$ which are the normal, the normal velocity and 
the mean curvature of $\Gamma^i(t) := \Gamma_{\rho^i,\mu^i}(t)$ at the point 
$\Phi^i(\sigma,t)$. Then we write equation (\ref{tripwithout}) over the fixed 
hypersurfaces $\G^1_\ast$, $\G^2_\ast$, and $\G^3_\ast$ as follows:
\begin{align} \label{triplefix}
 \left\{ \begin{array}{lll}
  \wV^i(\sigma,t) = \beta^i \wH^i(\sigma,t) 
  & \mbox{on } \; \G^i_\ast \;, \, t \in [0,T] \;, \; i=1,2,3 \;, \\
  \la \wN^1(\sigma,t), \wN^2(\sigma,t) \ra = \cos \theta^3 
  & \mbox{on } \; \Sigma_\ast \;, \, t \in [0,T] \;, \\
  \la \wN^2(\sigma,t), \wN^3(\sigma,t) \ra = \cos \theta^1 
  & \mbox{on } \; \Sigma_\ast \;, \, t \in [0,T] \;, \\
  (\rho^i(\sigma,0),\mu^i(\sigma,0)) = (\rho^i_0,\mu^i_0) 
  & \mbox{on } \; \G^i_\ast \times \Sigma_\ast \;, \; i=1,2,3 \;,
 \end{array} \right.
\end{align}
where we assume that the initial surfaces $\G^i_0$ from (\ref{tripwithout}) are 
given as 
$$
 \G^i_0 = \mbox{image}\{\sigma \mapsto 
 \Psi^i(\sigma,\rho^i_0(\sigma),\mu^i_0(\pr^i(\sigma))) 
 \;|\; \sigma \in \G^i_\ast\} \,.
$$
Herein we assume $\rho^i_0 \in C^{2+\alpha}(\G^i_\ast)$ with 
$\|\rho^i_0\|_{C^{2+\alpha}} \leq \varepsilon$ for some $\varepsilon > 0$,  
$\mu^i_0 \in C^{2+\alpha}(\Sigma_\ast)$ given by 
$\bomu_0 = \mathcal{T} \brho_0$ on $\Sigma_\ast$ 
and in addition the angle conditions from \eqref{tripwithout} for $\G^i_0$ 
shall be fulfilled. Furthermore, we assume that
\begin{equation}
\gamma^1 \beta^1 H_0^1+\gamma^2 \beta^2 H_0^2+\gamma^3 \beta^3 H_0^3 = 0\,\,\ 
\mbox{on}\,\ \Sigma_*,
\label{sum_H0}
\end{equation}
where $H_0^i$ is the mean curvature of $\G^i_0$. Note that equation 
\eqref{sum_H0} follows for smooth solutions from the first line in problem \eqref{triplefix} 
at $t=0$ on $\Sigma_\ast$, since for points on the triple junction 
we can write for the normal velocity $\wV^i = \la c'(0),\wN^i \ra$ with 
one curve $c:[0,t_0] \to \mathbb{R}^{n+1}$ on $\Sigma_\ast$ with 
$c(t) \in \Sigma(t)$ and use equation \eqref{younglaw3} for 
$\G^i_0$ which follows from the angle conditions. 

\begin{remark} \label{rem:initialnearref}
 The requirement that the $C^{2+\alpha}$-norm of the initial values 
$\rho^i_0$ is small implies that the initial hypersurfaces $\G^i_0$ are 
$C^{2+\alpha}$-close to the reference hypersurfaces $\G^i_\ast$, which 
are of class $C^{3+\alpha}$.
In order to make this compatible to condition \eqref{sum_H0}, there are 
two possibilities.

On the one hand we could start with initial hypersurfaces $\G^i_0$, 
which fulfill \eqref{sum_H0} and then choose hypersurfaces $\G^i_\ast$ 
of class $C^{3+\alpha}$, which are close enough to $\G^i_0$.
This would imply that condition \eqref{sum_H0} is almost fulfilled 
for $\G^i_\ast$ in the sense that 
$\left| \sum_{i=1}^3 \gamma^i \beta^i H^i_\ast \right|$ is small.

On the other hand we could additionally require condition \eqref{sum_H0} 
for the reference hypersurfaces $\G^i_\ast$. In this way the above approach 
would always work in the sense that there are hypersurfaces 
$\G^i_0$ given by $\rho^i_0$, such that $\|\rho^i_0\|_{C^{2+\alpha}}$ is 
small and \eqref{sum_H0} holds.
\end{remark}

Due to the condition $\theta^1 + \theta^2 + \theta^3 = 2 \pi$ and the fact 
that the surfaces all meet at a triple junction at their boundary, which follows 
from (\ref{meetcond}), the third angle condition 
\begin{equation} \label{triplewithoutfix}
  \begin{array}{rcll}
   \la \wN^2(\sigma,t), \wN^3(\sigma,t) \ra 
   = \cos \theta^2 
   \quad \mbox{on } \; \Sigma_\ast \;,\ t \in [0,T] \;, 
  \end{array}
\end{equation}
is automatically fulfilled and we omit it from now on. The equations 
(\ref{triplefix}) give a second order system of partial differential 
equations for the functions $(\rho^1,\mu^1,\rho^2,\mu^2,\rho^3,\mu^3)$.

More precisely, we can obtain the following representation for the equation. 
For the normal velocities we calculate
\begin{align*}
 \wV^i (\sigma,t) 
 &= \la \wN^i(\sigma,t), \partial_t \Phi^i(\sigma,t) \ra \\
 &= \la \wN^i(\sigma,t), 
     \partial_t \bigl\{ \sigma + \rho^i(\sigma,t) \, N^i_\ast(\sigma) 
     + \mu^i(\pr^i(\sigma),t) \, \tau^i_\ast(\sigma) 
     \bigr\} \ra \\
 &= \la \wN^i(\sigma,t), N^i_\ast(\sigma)\ra\, 
     \partial_t \rho^i(\sigma,t) 
     + \la \wN^i(\sigma,t), \tau^i_\ast(\sigma) \ra\, 
     \partial_t \mu^i(\pr^i(\sigma),t) \;. 
\end{align*}
We remark that there is a function $\tilde{N}^i$ such that 
$$
\wN^i(\sigma,t)
:=\tilde{N}^i(\sigma,\rho^i(\sigma,t),\mu^i(\pr^i(\sigma),t),
\nabla\rho^i(\sigma,t),\bar{\nabla}\mu^i(\pr^i(\sigma),t))
$$
is the unit normal vector field of $\Gamma^i(t)$, where 
$\nabla\rho^i$ is the gradient of $\rho^i$ on the hypersurfaces 
$\Gamma^i_\ast$, which is denoted in a local chart by 
$\nabla_j\rho^i=\partial_j\rho^i$\,($j=1,\ldots,n$), 
and $\bar{\nabla}\mu^i$ is the $(n-1)$-dimensional gradient of $\mu^i$ 
on a surface $\Sigma_\ast$. 
A formula for $\tilde{N}^i$ can be given with the help of 
a local chart through a normalized cross product of the tangential vectors 
$\partial_l \Phi^i$. Therefore $\tilde{N}^i$ is a nonlocal operator, 
since in its formula we find an expression 
$\mu^i(\pr^i(\sigma),t) \tau^i_\ast(\sigma)$ so that 
we do not only need $\brho$, $\bomu$ and its derivatives at the point $\sigma$ but also the point 
$\pr(\sigma) \in \partial \Gamma^i_\ast$ in order to calculate $\tilde{N}^i$. 

Since 
$$
(g^i)_{jk}=\la \partial_j\Phi^i, \partial_k\Phi^i \ra, \quad
(h^i)_{jk}=\la \wN^i, \partial_j\partial_k\Phi^i \ra,
$$
the mean curvature $\wH^i=(g^i)^{jk}(h^i)_{jk}$ is 
represented as
$$
\wH^i(\sigma,t) := \tilde{H}^i(\sigma,\rho^i(\sigma,t),\mu^i(\pr^i(\sigma),t),
\nabla\rho^i(\sigma,t),\bar{\nabla}\mu^i(\pr^i(\sigma),t),
\nabla^2\rho^i(\sigma,t),\bar{\nabla}^2\mu^i(\pr^i(\sigma),t)),
$$
where $\nabla^2\rho^i$ is the Hessian of $\rho^i$ on hypersurfaces 
$\Gamma_\ast^i$ defined in a local chart by 
\[ \nabla^2_{j_1j_2}\rho^i
=\nabla_{j_1}\nabla_{j_2}\rho^i=\partial_{j_1}\partial_{j_2}\rho^i
-\Gamma_{j_1j_2}^k\partial_k \rho^i \quad (j_1,j_2 = 1,\ldots,n), \]
where $\Gamma_{j_1j_2}^k$ are the Christoffel symbols for $\Gamma_*^i$ and 
we used the sum convention for the last term. 
The expression $\bar{\nabla}^2\mu^i$ denotes the Hessian of $\mu^i$ on 
the $(n-1)$-dimensional surface $\Sigma_\ast$. Note that the coefficients in front of 
the term $\bar{\nabla}_j\bar{\nabla}_k\,\mu^i$ in $\wH^i$ are given by
$$
(g^i)^{jk} \la \tau_*^i, \wN^i \ra.
$$
Thus the mean curvature flow equations can be reformulated as
\begin{equation}
\partial_t \rho^i
=a^i(\sigma,\rho^i,\mu^i)H^i(\sigma,\rho^i,\mu^i) 
+a^i_{\dag}(\sigma,\rho^i,\mu^i)\mu^i_t,
\label{rho_eq}
\end{equation}
where $H^i(\sigma,\rho^i,\mu^i):=\tilde{H}^i(\sigma,\rho^i,\mu^i,\nabla\rho^i,
\bar{\nabla}\mu^i, \nabla^2\rho^i, \bar{\nabla}^2\mu^i)$ and 
\begin{align*}
&a^i(\sigma,\rho^i,\mu^i)
:=\tilde{a}^i(\sigma,\rho^i, \mu^i, \nabla \rho^i, \bar{\nabla} \mu^i)
=\frac{\beta^i}{\la N_*^i(\sigma), 
\tilde{N}^i(\sigma, \rho^i, \mu^i, \nabla\rho^i, \bar{\nabla}\mu^i) \ra}, \\
&a^i_{\dag}(\sigma,\rho^i,\mu^i)
:=\tilde{a}_{\dag}^i(\sigma,\rho^i,\mu^i,\nabla\rho^i,\bar{\nabla}\mu^i)
=- \frac{\la \tau_*^i(\sigma), 
\tilde{N}^i(\sigma, \rho^i, \mu^i, \nabla\rho^i, \bar{\nabla} \mu^i) \ra}
{\la N_*^i(\sigma), 
\tilde{N}^i(\sigma, \rho^i, \mu^i, \nabla\rho^i, \bar{\nabla}\mu^i) \ra} \,.
\end{align*}
Note that we omitted the mapping $\pr^i$ in the functions $\mu^i$ 
for reasons of shortness.

Now we will write equation \eqref{rho_eq} as an evolution equation, 
which is nonlocal in space, solely for the mappings $\rho^i$ by using 
the linear dependence \eqref{equ:lindep} on $\Sigma_\ast$. 
To this end, we use \eqref{equ:lindep} in the form 
$\mu^i =(\left. \mathcal{T} \brho \right|_{\Sigma_\ast} )^i$ 
and rewrite \eqref{rho_eq} into
\begin{align} \label{rho_equ_rewritten}
\partial_t \rho^i 
= \mathcal{F}^i(\rho^i,\left.\brho \right|_{\Sigma_\ast}) 
  + \mathfrak{a}_{\dag}^i(\rho^i,\left.\brho \right|_{\Sigma_\ast}) \partial_t 
    (\mathcal{T} \brho \circ \pr^i)^i \,,
\end{align}
where (omitting the $t$-variable for the moment)
\begin{align*}
\mathcal{F}^i(\rho^i,\left.\brho \right|_{\Sigma_\ast}) (\sigma) 
&= a^i(\sigma,\rho^i,(\left. \mathcal{T} \brho \right|_{\Sigma_\ast} )^i)
   H^i(\sigma,\rho^i,(\left. \mathcal{T} \brho \right|_{\Sigma_\ast} )^i) 
&\quad \mbox{for } \, \sigma \in \G^i_\ast \,, \\
\mathfrak{a}_{\dag}^i(\rho^i,\left.\brho \right|_{\Sigma_\ast}) (\sigma) 
&= a^i_\dag(\sigma,\rho^i,(\left. \mathcal{T} \brho \right|_{\Sigma_\ast} )^i)
&\quad \mbox{for } \, \sigma \in \G^i_\ast \,.
\end{align*}
With the following notations on $\Sigma_\ast$ given by
\begin{align*}
\mathcal{F}(\brho,\left.\brho \right|_{\Sigma_\ast})(\sigma) 
&= \left( \mathcal{F}^i(\rho^i,\left.\brho \right|_{\Sigma_\ast}) (\sigma) 
   \right)_{i=1,2,3}  
&& \mbox{for } \, \sigma \in \Sigma_\ast \,, \\
\mathcal{D}_\dag(\brho,\left.\brho \right|_{\Sigma_\ast})(\sigma) 
&= \operatorname{diag}\left( \left( 
   \mathfrak{a}^i_\dag(\rho^i,\left.\brho \right|_{\Sigma_\ast}) (\sigma) 
   \right)_{i=1,2,3} \right)
&& \mbox{for } \, \sigma \in \Sigma_\ast 
\end{align*}
we can write \eqref{rho_equ_rewritten} as vector identity on $\Sigma_\ast$ 
through
\begin{align} \label{rho_equ_vector}
\partial_t \brho 
= \mathcal{F}(\brho,\left.\brho \right|_{\Sigma_\ast}) 
  + \mathcal{D}_\dag(\brho,\left.\brho \right|_{\Sigma_\ast}) 
    \mathcal{T}(\partial_t\brho) \,.
\end{align}
Rearranging leads to 
$$
\left( Id - \mathcal{D}_\dag(\brho,\left.\brho \right|_{\Sigma_\ast}) 
\mathcal{T} \right) \partial_t \brho 
= \mathcal{F}(\brho,\left.\brho \right|_{\Sigma_\ast}) 
\quad \mbox{on } \, \Sigma_\ast \,.
$$
Then, with the help of $\mathcal{P}(\brho,\left.\brho \right|_{\Sigma_\ast})$ 
given by
\begin{align} \label{mathcal_P}
\mathcal{P}(\brho,\left.\brho \right|_{\Sigma_\ast})
:= {\cal T} \left( 
   Id - \mathcal{D}_\dag(\brho,\left.\brho \right|_{\Sigma_\ast}) 
   {\cal T} \right)^{-1},
\end{align}
it follows that 
\begin{align*}
\mathcal{T} \partial_t \brho 
= \mathcal{P}(\brho,\left.\brho \right|_{\Sigma_\ast}) 
  \mathcal{F}(\brho,\left.\brho \right|_{\Sigma_\ast}) 
\quad \mbox{on } \; \Sigma_\ast \,. 
\end{align*}
In a neighbourhood of $\Sigma_\ast$, where $\pr^i$ is defined, this leads to
$$
\partial_t \mu^i(\pr^i(\sigma)) 
= ({\cal T} \partial_t \brho(\pr^i(\sigma)) )^i 
= \left( \left\{ \mathcal{P}(\brho,\left.\brho \right|_{\Sigma_\ast}) 
  \mathcal{F}(\brho,\left.\brho \right|_{\Sigma_\ast}) \right\} \circ \pr^i 
  \right)^i \,. 
$$
Hence, the equation (\ref{rho_eq}) is rewritten as
$$
\partial_t \rho^i 
= \mathcal{F}^i(\rho^i,\left.\brho \right|_{\Sigma_\ast})
  + \mathfrak{a}_{\dag}^i(\rho^i,\left.\brho \right|_{\Sigma_\ast}) 
    \left( \left\{ \mathcal{P}(\brho,\left.\brho \right|_{\Sigma_\ast}) 
    \mathcal{F}(\brho,\left.\brho \right|_{\Sigma_\ast}) \right\} \circ \pr^i 
    \right)^i
\quad \mbox{on } \; \G^i_\ast \,.
$$
The second term of the right hand side of this equation contains non-local 
terms including the highest order derivatives, that is, the second order 
derivatives. 

The angle conditions at the triple junction $\Sigma_\ast$ can be written as
\begin{align*}
\mathcal{G}^2(\brho) 
:= \la \mathcal{N}^1(\brho) , \mathcal{N}^2(\brho) \ra - \cos \theta^3 
&= 0 \quad \mbox{ on } \, \Sigma_\ast \,, \; t \geq 0 \,, \\
\mathcal{G}^3(\brho) 
:= \la \mathcal{N}^2(\brho) , \mathcal{N}^3(\brho) \ra - \cos \theta^1 
&= 0 \quad \mbox{ on } \, \Sigma_\ast \,, \; t \geq 0 
\end{align*}
with the notation 
${\cal N}^i(\bmv)(\s,t):=\tilde{N}^i(\sigma,v^i(\s,t),
({\cal T}(\bmv\circ\pr(\s,t)))^i,\nabla v^i(\s,t),
\bar{\nabla}({\cal T}(\bmv\circ\pr(\s,t)))^i)$. 
Note that due to $\sigma = \pr^i(\sigma)$ for $\sigma \in \Sigma_\ast$ 
the operators $\mathcal{G}^1$ and $\mathcal{G}^2$ are local differential 
operators and $\mathcal{G}^2$ depends only on $\rho^1$ and $\rho^2$ 
as well as $\mathcal{G}^3$ only on $\rho^2$ and $\rho^3$. 

Finally we have to take care of the equations \eqref{fourcond}, 
which are needed to make sure that the attachment condition 
\eqref{meetcond} holds. Equation \eqref{fourcond}$(ii)$ is already included 
implicitly, so that we are left with \eqref{fourcond}$(i)$ given by
\begin{align*}
\mathcal{G}^1(\brho) 
:= \gamma^1 \rho^1 + \gamma^2 \rho^2 + \gamma^3 \rho^3 = 0 
\,\,\ \mbox{on}\,\,\ \Sigma_\ast \,,\ t \geq 0 \,.
\end{align*}
Altogether this leads to the following nonlinear, nonlocal problem for 
$i=1,2,3$:
\begin{align} \label{prob:nonlinearnonlocal}
\left\{ \begin{array}{ll}
\partial_t \rho^i 
= \mathcal{F}^i(\rho^i,\left.\brho \right|_{\Sigma_\ast})
  + \mathfrak{a}_{\dag}^i(\rho^i,\left.\brho \right|_{\Sigma_\ast}) 
    \left( \left\{ \mathcal{P}(\brho,\left.\brho \right|_{\Sigma_\ast}) 
    \mathcal{F}(\brho,\left.\brho \right|_{\Sigma_\ast}) \right\} \circ \pr^i
    \right)^i
  & \mbox{on}\,\,\ \G^i_\ast \;,\ t \geq 0 \,, \\
\mathcal{G}^i(\brho) = 0 
  & \mbox{on}\,\,\ \Sigma_\ast \;,\ t \geq 0 \,, \\
\rho^i(.\,,0) = \rho^i_0  
  & \mbox{on}\,\,\ \G^i_\ast\,.
\end{array} \right.
\end{align}

\subsection{The compatibility conditions} \label{subsec:compcond}

For $\rho^i_0$ we assume the compatibility conditions
\begin{equation} \label{comp:rho_0}
\mathcal{G}^i(\brho_0) = 0 
\,\,\ \mbox{on}\,\,\ \Sigma_\ast \quad \mbox{ and } \quad
\sum_{i=1}^3 \gamma^i 
\mathcal{K}^i(\rho^i_0,\left.\brho_0 \right|_{\Sigma_\ast}) = 0 
\,\,\ \mbox{on}\,\,\ \Sigma_\ast \,,
\end{equation}
where $\mathcal{K}^i$ denotes the right side of the first line 
in \eqref{prob:nonlinearnonlocal}. 
To state all the dependencies explicitly, we remark that by construction 
there is a function $\widetilde{\mathcal{K}}^i$ such that
\begin{align} \label{equ:tildeK}
\mathcal{K}^i(\rho^i,\left.\brho \right|_{\Sigma_\ast})(\sigma,t) 
= \widetilde{\mathcal{K}}^i \big(\sigma,&\,\rho^i(\sigma,t),
  \nabla \rho^i(\sigma,t),\nabla^2 \rho^i(\sigma,t),
  \left.\brho \right|_{\Sigma_\ast}(\pr^i(\sigma),t),\ldots \nonumber \\
& \ldots, \bar{\nabla} \left.\brho \right|_{\Sigma_\ast}(\pr^i(\sigma),t),
  \bar{\nabla}^2 \left.\brho \right|_{\Sigma_\ast}(\pr^i(\sigma),t) \big) \,. 
\end{align}

Note that we always set $\bomu_0 = \mathcal{T} \brho_0$ on $\Sigma_\ast$ and 
therefore the geometric compatibility condition \eqref{sum_H0} is fulfilled 
since we require \eqref{comp:rho_0} for $\brho_0$. 
This is stated in the following lemma.

\begin{lemma} \label{lem:geomcompcond}
The compatibility conditions \eqref{comp:rho_0} for $\brho_0$ imply 
the geometric compatibility condition \eqref{sum_H0}.
\end{lemma}
\begin{proof}
Using the abbreviations $\mathcal{K}^i_0 
= \mathcal{K}^i(\rho^i_0,\left.\brho_0 \right|_{\Sigma_\ast})$ and 
$\mathcal{L}^i_0 = \left( \mathcal{T} \mathcal{K}_0 \right)^i$, where 
$\mathcal{K}_0=\left( \mathcal{K}^i_0 \right)_{i=1,2,3}$, 
we get from the second compatibility condition in \eqref{comp:rho_0} with 
arguments similar as in the proof of Lemma~\ref{lem:equmeetcond} 
(see \cite{DG11}) that
\begin{align*}
\mathcal{K}^i_0 N^i_\ast + \mathcal{L}^i_0 \tau^i_\ast 
= \mathcal{K}^j_0 N^j_\ast + \mathcal{L}^j_0 \tau^j_\ast 
\quad \mbox{on } \, \Sigma_\ast \,.
\end{align*}
Now we show on $\Sigma_\ast$ the following identity
\begin{align} \label{ident:help}
\la \left( \mathcal{K}^i_0 N^i_\ast + \mathcal{L}^i_0 \tau^i_\ast \right), 
N^i_0 \ra = \beta^i H^i_0 \quad \mbox{on } \, \Sigma_\ast \,.
\end{align}
To see this, we write in the following an index $0$ on every term to 
indicate evaluation at $\brho_0$ to get
\begin{align*}
\mathcal{K}^i_0 
&= a^i_0 H^i_0 + a^i_{\dagger,0} 
   \left( \mathcal{T} (Id - \mathcal{D}_{\dagger,0} \mathcal{T})^{-1} 
   \mathcal{F}_0 \right)^i \\
&= a^i_0 H^i_0 + \left( \mathcal{D}_{\dagger,0} \mathcal{T} 
   (Id - \mathcal{D}_{\dagger,0} \mathcal{T})^{-1} \mathcal{F}_0 \right)^i \,, 
\\
\mathcal{K}_0 
&= \mathcal{F}_0 + \mathcal{D}_{\dagger,0} \mathcal{T} 
   (Id - \mathcal{D}_{\dagger,0} \mathcal{T})^{-1} \mathcal{F}_0 \,,
\end{align*}
respectively. With the definition of $a^i$ and $a^i_\dagger$ this leads to 
\begin{align*}
\mathcal{K}^i_0 \la N^i_\ast , N^i_0 \ra 
&= \beta^i H^i_0 - \la \tau^i_\ast , N^i_0 \ra 
   \left( \mathcal{T} (Id - \mathcal{D}_{\dagger,0} \mathcal{T})^{-1} 
   \mathcal{F}_0 \right)^i.
\end{align*}
In order to obtain \eqref{ident:help} it is therefore enough to show that
$$
- \la \tau^i_\ast , N^i_0 \ra 
\left( \mathcal{T} (Id - \mathcal{D}_{\dagger,0} \mathcal{T})^{-1} 
\mathcal{F}_0 \right)^i
= - \mathcal{L}^i_0 \la \tau^i_\ast , N^i_0 \ra \; 
$$
which is, without loss of generality, equivalent to
$$
\mathcal{T} (Id - \mathcal{D}_{\dagger,0} \mathcal{T})^{-1} \mathcal{F}_0 
= \mathcal{T} \mathcal{K}_0  \,.
$$
To obtain the last equality we observe that
\begin{align*}
(Id - \mathcal{D}_{\dagger,0} \mathcal{T})^{-1} \mathcal{F}_0 - \mathcal{K}_0 
&= (Id - \mathcal{D}_{\dagger,0} \mathcal{T})^{-1} \mathcal{F}_0 
   - \mathcal{F}_0 - \mathcal{D}_{\dagger,0} \mathcal{T} 
     (Id - \mathcal{D}_{\dagger,0} \mathcal{T})^{-1} \mathcal{F}_0 \\
&= (Id - \mathcal{D}_{\dagger,0} \mathcal{T})(Id - \mathcal{D}_{\dagger,0} 
   \mathcal{T})^{-1} \mathcal{F}_0  - \mathcal{F}_0 \\
&= 0 \,,
\end{align*}
so that finally \eqref{ident:help} is verified. 

Since the term in brackets on the left side of \eqref{ident:help} is 
independent of $i$, we can multiply by $\gamma^i$, sum over $i=1,2,3$ and 
use $\eqref{sum_N}$ resulting from the angle conditions for $\G^i_0$ to 
derive finally equation \eqref{sum_H0}, that is 
$\sum_{i=1}^3 \gamma^i \beta^i H^i_0 = 0$ on $\Sigma_\ast$. 
\end{proof}

\section{Linearization} \label{sec:linearization}

In this section we will derive the linearization of the nonlinear nonlocal 
problem \eqref{prob:nonlinearnonlocal} around $\brho \equiv \mathbf{0}$, 
that is around the fixed reference hypersurfaces $\G^i_\ast$. This will 
be done by considering the geometric problem \eqref{triplefix} and 
linearize this around $(\brho,\boldsymbol{\mu}) \equiv \mathbf{0}$. 
For this part we can use the work of Depner and Garcke \cite{DG11}, 
where the authors considered stationary reference hypersurfaces, 
and comment on the differences. To explain our notation we give 
the calculations for the normal velocity and just refer for 
the linearization of the mean curvature and the angle conditions 
to \cite{DG11}. In each term in \eqref{triplefix}, we write 
$\epsi u^i$ and $\epsi \phi^i$ instead of $\rho^i$ and $\mu^i$ for 
$i=1,2,3$, differentiate with respect to $\epsi$, and set $\epsi=0$ 
in the resulting equations. Here, we have to assume the triple junction condition 
\eqref{meetcond} for $\Phi^i_{u^i,\phi^i}$, which is nothing else than 
assuming it for $\Phi^i_{\epsi u^i, \epsi \phi^i}$. In this way, we will get 
linear partial differential equations, where we then express terms of 
$\phi^i$ as nonlocal terms in $\bmu$ with the help of \eqref{equ:lindep} 
for $\bmu$ and $\boldsymbol{\phi}$.

\medskip

\noindent {\bf Linearization of the normal velocity:}
For the linearization of the normal velocity $\wV^i$, we obtain
\begin{align*}
 & \left. \frac{d}{d \epsi} V^i \circ \Phi^i_{\epsi u^i, \epsi \phi^i}
   (\sigma,t) \right|_{\epsi=0} \\
 &=\frac{d}{d\epsi}\bigl\{
   \la N^i\circ\Phi^i_{\epsi u^i,\epsi\phi^i}, N^i_* \ra
    \,\partial_t(\epsi u^i) 
    +\la N^i\circ\Phi^i_{\epsi u^i,\epsi\phi^i}, \tau^i_* \ra
    \,\partial_t(\epsi\phi^i)\bigr\}\biggr|_{\epsi=0} \\[0.1cm]%
 &= \underbrace{\la N^i_\ast(\sigma), N^i_\ast(\sigma) \ra}_{= 1} \, 
     \partial_t u^i(\sigma,t) 
     + \underbrace{\la N^i_\ast(\sigma), \tau^i_\ast(\sigma) \ra}_{= 0} \, 
     \partial_t \phi^i(\pr^i(\sigma),t) \\
 &= \partial_t u^i(\sigma,t) \,.
\end{align*}

\medskip

\noindent {\bf Linearization of the mean curvature:}
For the linearization of the mean curvature $\wH^i$, we use 
the following result, see Depner, Garcke \cite{DG11} and Depner \cite{Dep10}, 
where \cite{Dep10} contains the detailed calculation:
\begin{align*}
 & \left. \frac{d}{d \epsi} 
    H^i \circ \Phi^i_{\epsi u^i, \epsi \phi^i}(\sigma,t) 
    \right|_{\epsi=0} \\
 &= \Delta_{\G^i_\ast} u^i(\sigma,t) 
     + |\Pi^i_\ast|^2(\sigma,t) u^i(\sigma,t)
     + \la \grad_{\G^i_\ast} H^i(\sigma), 
       \left[ \left. \frac{d}{d \epsi} 
       \Phi^i_{\epsi u^i, \epsi \phi^i}(\sigma,t) 
       \right|_{\epsi=0} \right]^T \ra \;,
\end{align*}
where $\Delta_{\Gamma^i_\ast}$ is the Laplace-Beltrami operator on 
$\Gamma^i_\ast$, $\Pi^i_\ast$ denotes the second fundamental form of 
$\G^i_\ast$ and $|\Pi^i_\ast|^2$ is the squared norm of $\Pi^i_\ast$ and hence 
given as the sum of the squared principal curvatures. Furthermore 
$\grad_{\G^i_\ast}$ is the surface gradient on $\G^i_\ast$ and 
$[\,\cdot\,]^T$ is the tangential part of a vector. 
Note that the last term would vanish for reference hypersurfaces with 
constant mean curvature. For the last term we compute
\begin{align*}
 \left[ \left. \frac{d}{d \epsi} 
 \Phi^i_{\epsi u^i, \epsi \phi^i}(\sigma,t) \right|_{\epsi=0} \right]^T 
 =&\, \left[ \left. \frac{d}{d \epsi} 
 \left( \sigma + \epsi u^i(\sigma,t) \, N^i_\ast(\sigma) 
 + \epsi \phi^i(\pr^i(\sigma),t) \, \tau^i_\ast(\sigma)
 \right) \right|_{\epsi=0} \right]^T \\
 =&\, \left[ u^i(\sigma,t)  \, N^i_\ast(\sigma) 
 + \phi^i(\pr^i(\sigma),t) \, \tau^i_\ast(\sigma) \rule{0cm}{0,35cm} \right]^T 
 \\[0.1cm]%
 =&\, \phi^i(\pr^i(\sigma),t) \, \tau^i_\ast(\sigma) \,,
\end{align*}
so that we get
\begin{align*}
 &\left. \frac{d}{d \epsi} H^i \circ 
 \Phi^i_{\epsi u^i, \epsi \phi^i}(\sigma,t) \right|_{\epsi=0} \\[0.1cm]%
 &= \Delta_{\G^i_\ast} u^i(\sigma,t) 
 + |\Pi^i_\ast|^2(\sigma,t) u^i(\sigma,t) 
 + \la \grad_{\G^i_\ast} H^i(\sigma), \tau^i_\ast(\sigma) \ra\,
   \phi^i(\pr^i(\sigma),t) \,. 
\end{align*}

\medskip

\noindent {\bf Linearization of the angle conditions:}
The linearization of the angle condition $\la \wN^i,\wN^j \ra = \cos \theta^k$ 
is the technically most challenging part and 
we use the following result of Depner and Garcke \cite{DG11}:
\begin{align*}
 & \left. \frac{d}{d \epsi} 
   \la N^i \circ \Phi^i_{\epsi u^i, \epsi \phi^i}, 
   N^j \circ \Phi^j_{\epsi u^j, \epsi \phi^j} \ra 
   \right|_{\epsi=0} \\
 &= \partial_{\nu^i_\ast} u^i + \Pi^i_\ast(\nu^i_\ast,\nu^i_\ast) \phi^i
    - \partial_{\nu^j_\ast} u^j - \Pi^j_\ast(\nu^j_\ast,\nu^j_\ast) \phi^j 
\end{align*}
on $\Sigma_\ast$ for $t \geq 0$ and for $(i,j)=(1,2)$ and $(2,3)$.
Note that in \cite{DG11} there was a second equivalent formulation of 
the above formula, which is not possible here, since the 
reference hypersurfaces are not stationary. 
Nevertheless with the help of \eqref{equ:lindep} we can get rid of 
$\phi^i$ by expressing it with the help of $\bmu$. 

Altogether, we get for the linearization of (\ref{triplefix}) 
the following linear system of partial differential equations for 
$(u^i,\phi^i)$ and $i=1,2,3$.
\begin{equation} \label{triplinearized}
 \left\{
 \begin{array}{ll}
  \partial_t u^i 
  = \beta^i \left( \Delta_{\G^i_\ast} u^i + |\Pi^i_\ast|^2 u^i \right)
    + \beta^i \la \grad_{\G^i_\ast} H^i, \tau^i_\ast \ra \, 
      (\phi^i\circ\pr^i) \quad
  & \mbox{on } \G^i_\ast \times [0,T] \;, \\[0.1cm]
  \gamma^1 u^1 + \gamma^2 u^2 + \gamma^3 u^3 = 0 
  & \mbox{on } \Sigma_\ast \times [0,T]  \,, \\[0.1cm]
  \partial_{\nu^1_\ast} u^1
    + \Pi^1_\ast(\nu^1_\ast,\nu^1_\ast) \phi^1
  = \partial_{\nu^2_\ast} u^2 
    + \Pi^2_\ast(\nu^2_\ast,\nu^2_\ast) \phi^2 
  & \mbox{on } \Sigma_\ast \times [0,T]  \,, \\[0.1cm]
  \partial_{\nu^2_\ast} u^2 
    + \Pi^2_\ast(\nu^2_\ast,\nu^2_\ast) \phi^2 
  = \partial_{\nu^3_\ast} u^3 
    + \Pi^3_\ast(\nu^3_\ast,\nu^3_\ast) \phi^3 
  & \mbox{on } \Sigma_\ast \times [0,T] \,, \\[0.1cm]
  (u^i,\phi^i)\big|_{t=0} = (\rho^i_0,\mu^i_0)
  & \mbox{on } \G^i_\ast \,.
 \end{array}
 \right.
\end{equation}
Note that $\phi^i \circ \pr^i$ can be rewritten as 
$\left(\mathcal{T} (\bmu \circ \pr^i) \right)^i$ due to equation 
\eqref{equ:lindep}, which also has to hold for $\boldsymbol{\phi}$ and $\bmu$. 
Now we are able to rewrite the nonlinear, nonlocal problem 
\eqref{prob:nonlinearnonlocal} as a perturbation of 
a linearized problem. Let the operator ${\cal A}^i$ and 
the function $\zeta^i$ be given by
$$
{\cal A}^i 
= \beta^i \bigl\{\Delta_{\Gamma_\ast^i}+|\Pi_\ast^i|^2 I \bigr\}, \quad
\zeta^i(\sigma) 
= \beta^i \la\grad_{\G^i_\ast} H^i_\ast(\s) , \tau^i_\ast(\s) \ra \,.
$$
We also introduce an operator corresponding to the linearized boundary 
conditions given by
\begin{align*}
 \dis\sum_{j=1}^3{\cal B}^{ij}u^j 
 &= \left\{\begin{array}{ll}
     \gamma^1u^1+\gamma^2u^2+\gamma^3u^3,&i=1, \\[0.1cm]
     \la\nabla_{\Gamma_\ast^1}u^1,\nu_\ast^1\ra 
      +\dfrac{\kappa_\ast^1}{s^1}(c^2u^2-c^3u^3)
      -\bigl\{\la\nabla_{\Gamma_\ast^2}u^2,\nu_\ast^2\ra 
      +\dfrac{\kappa_\ast^2}{s^2}(c^3u^3-c^1u^1)\bigr\},&i=2, \\[0.1cm]
     \la\nabla_{\Gamma_\ast^2}u^2,\nu_\ast^2\ra 
      +\dfrac{\kappa_\ast^2}{s^2}(c^3u^3-c^1u^1)
      -\bigl\{\la\nabla_{\Gamma_\ast^3}u^3,\nu_\ast^3\ra 
      +\dfrac{\kappa_\ast^3}{s^3}(c^1u^1-c^2u^2)\bigr\},&i=3,
    \end{array}\right.
\end{align*}
where $\kappa_\ast^i:= \Pi^i_\ast(\nu^i_\ast,\nu^i_\ast)$ denotes 
the normal curvature of $\G^i_\ast$ in direction of $\nu^i_\ast$. 

With this notation we can rewrite the nonlinear nonlocal problem 
\eqref{prob:nonlinearnonlocal} into the following one, where $i=1,2,3$:
\begin{equation}
  \left\{\begin{array}{ll}
           \partial_t u^i={\cal A}^iu^i+\zeta^i({\cal T}(\bmu\circ\pr^i))^i
              +\mathfrak{f}^i(u^i,\left.\bmu\right|_{\Sigma_\ast})
              & \mbox{on } \G^i_\ast \times[0,T], \\[0.1cm]
           \dis\sum_{j=1}^3{\cal B}^{ij}u^j=\mathfrak{b}^i(\bmu) 
              & \mbox{on } \Sigma_\ast \times [0,T], \\[0.6cm]
           u^i\big|_{t=0}=\rho^i_0 & \mbox{on } \G^i_\ast \,.
\end{array}\right.
\label{L_eq}
\end{equation}
Herein, $\mathfrak{f}^i$ and $\mathfrak{b}^i$ are defined through
\begin{align}
 \begin{split}
  \mathfrak{f}^i(v^i,\left.\bmv\right|_{\Sigma_\ast}) 
:={}& {\cal F}^i(v^i,\left.\bmv\right|_{\Sigma_\ast}) -\Bigl\{{\cal A}^iv^i
    +\zeta^i({\cal T}(\bmv\circ\pr^i))^i\Bigr\}  \\ 
  {}& +\mathfrak{a}_{\dag}^i(v^i,\left.\bmv\right|_{\Sigma_\ast}) 
     ( \{{\cal P}(\bmv,\left.\bmv\right|_{\Sigma_\ast})
     {\cal F}(\bmv,\left.\bmv\right|_{\Sigma_\ast})\} \circ \pr^i )^i,  \label{equ:defoff}
 \end{split} \\
 \mathfrak{b}^i(\bmv) 
   :={}&  -\Bigl\{{\cal G}^i(\bmv)-\dis\sum_{j=1}^3{\cal B}^{ij}v^j\Bigr\} \,. \label{equ:defofb}
\end{align}
Note that the first boundary condition on the triple junction $\Sigma_\ast$ 
in problem \eqref{prob:nonlinearnonlocal} is already linear and therefore 
$\mathfrak{b}^1(\bmv) \equiv 0$. But we will nevertheless use 
$\mathfrak{b}^1$ to avoid some case by case analysis. 

\section{Analysis of the linearized problem}\label{sec:lin}

In this section we consider the linear nonhomogeneous problem corresponding 
to \eqref{L_eq}. We will give a local existence result for the case 
with initial data zero and then outline the necessary steps for 
the arbitrary case. First we introduce for an arbitrary smooth 
Riemannian manifold $(\G,g)$ some notation. 
For an integer $k$ and smooth functions $u\,:\,\Gamma\to{\mathbb R}$, 
we denote by $\nabla^k u$ the $k$-th covariant derivative of $u$ and 
by $|\nabla^k u|$ the norm of $\nabla^k u$ defined in a local chart by, 
see e.g. \cite{Aub82},
$$
|\nabla^k u|^2=g^{i_1j_1}\cdots g^{i_kj_k}(\nabla^k_{i_1\ldots i_k} u)
(\nabla^k_{j_1\ldots j_k} u).
$$
Note that $\nabla_i u=\partial_i u$ and 
$\nabla^2_{i_1i_2} u=\nabla_{i_1}\nabla_{i_2}u
=\partial_{i_1}\partial_{i_2}u-\Gamma_{i_1i_2}^m\partial_mu$. 
For $T>0$ and $0<\alpha<1$, set $Q_T=\Gamma\times[0,T]$ and 
\begin{align*}
&\|u\|_{\infty}=\sup_{(\sigma,t)\in Q_T}|u(\sigma,t)|, \\
&\la u\ra_\sigma^{\alpha}
=\sup_{(\sigma,t),\,(\tilde{\sigma},t)\in Q_T,\,\sigma\ne\tilde{\sigma}}
\frac{|u(\sigma,t)-u(\tilde{\sigma},t)|}
{\{d_g(\sigma,\tilde{\sigma})\}^{\alpha}}, \\
&\la u\ra_t^{\alpha}
=\sup_{(\sigma,t),\,(\sigma,\tilde{t}\,)\in Q_T,\,t\ne\tilde{t}}
\frac{|u(\sigma,t)-u(\sigma,\tilde{t})|}{|t-\tilde{t}\,|^{\alpha}},
\end{align*}
where $d_g$ denotes the distance on $\G$ induced by the metric $g$. 
Then, we define the norms $\|u\|_{C^{\alpha,\frac{\alpha}2}(Q_T)}$ and 
$\|u\|_{C^{2+\alpha,1+\frac{\alpha}2}(Q_T)}$ as 
\begin{align*}
&\|u\|_{C^{\alpha,\frac{\alpha}2}(Q_T)}=\|u\|_{\infty}+\la u\ra_x^{\alpha}
+\la u\ra_t^{\alpha}, \\
&\|u\|_{C^{2+\alpha,1+\frac{\alpha}2}(Q_T)}
=\|u\|_{\infty}+\|\nabla u\|_{\infty}
+\|\nabla^2 u\|_{C^{\alpha,\frac{\alpha}2}(Q_T)}
+\|\partial_t u\|_{C^{\alpha,\frac{\alpha}2}(Q_T)}.
\end{align*}

Set ${\cal X}_T=C^{2+\alpha,1+\frac{\alpha}2}(Q^1_T)\times 
C^{2+\alpha,1+\frac{\alpha}2}(Q^2_T)\times 
C^{2+\alpha,1+\frac{\alpha}2}(Q^3_T)$, 
where $Q^i_T=\Gamma_\ast^i\times[0,T]$.
Then we have the following theorem about existence of solutions to 
the linearized, nonhomogeneous problem with initial data zero. 

\begin{theo}\label{thm:L_exist}
Let $\alpha\in(0,1)$. Then there exists a $\delta_0>0$ such that 
for every $f^i\in C^{\alpha,\frac{\alpha}2}(Q^i_{\delta_0})$ 
and $b^i\in C^{1+\alpha,\frac{1+\alpha}2}(\Sigma_\ast\times[0,\delta_0])$, 
$i=1,2,3$, with $b^1 \equiv 0$ and which fulfill the compatibility condition
$$
(\gamma^1f^1+\gamma^2f^2+\gamma^3f^3)\big|_{t=0}=0,\,\,\ 
b^i\big|_{t=0}=0\,\,\ \mbox{on}\,\ \Sigma_\ast, \; i=2,3 \,,
$$
the problem
\begin{equation}
\left\{\begin{array}{ll}
\partial_t u^i={\cal A}^iu^i+\zeta^i({\cal T}(\bmu\circ\pr^i))^i+f^i
& \mbox{on } \; \G^i_\ast\times[0,T], \\
\dis\sum_{j=1}^3{\cal B}^{ij}u^j=b^i 
& \mbox{on } \; \Sigma_\ast\times[0,T], \\[0.6cm]
u^i\big|_{t=0}=0 & \mbox{on } \; \G^i_\ast
\end{array}\right.
\label{L_eq_1}
\end{equation}
for $i=1,2,3$ has a unique solution $(u^1,u^2,u^3)\in{\cal X}_{\delta_0}$. 
Moreover, there exists a $C>0$, which is independent of $\delta_0$, such that
$$
\sum_{i=1}^3\|u^i\|_{C^{2+\alpha,1+\frac{\alpha}2}(Q^i_{\delta_0})}
\le C\sum_{i=1}^3\bigl\{\|f^i\|_{C^{\alpha,\frac{\alpha}2}(Q^i_{\delta_0})}
+\|b^i\|_{C^{1+\alpha,\frac{1+\alpha}2}(\Sigma_\ast\times[0,\delta_0])}\bigr\}. 
$$
\end{theo}

First, we will consider problem \eqref{L_eq_1} without the nonlocal term 
$\zeta^i (\mathcal{T}(u \circ \pr^i))^i$ and at the end we will include it
with the help of a perturbation argument. 

In order to apply the $C^\alpha$-regularity theory of Solonnikov \cite{Sol} 
we need to show that the boundary value problem \eqref{L_eq_1} fulfills the 
Lopantinskii-Shapiro compatibility conditions, see Chapter I of \cite{Sol},
where the conditions are stated. 
To this end we have to rewrite problem \eqref{L_eq_1} with the help of local
coordinates and a partition of unity as a problem in Euclidean space. We will
do this locally around the triple junction with specifically chosen local 
coordinates, since the compatibility conditions have to be checked just there.
Locally around a point 
$\sigma \in \Sigma_\ast$ we choose for each of the surfaces $\Gamma^i_\ast$, 
$i=1,2,3$, local coordinates $(x_1,\ldots,x_n)$ such that 
$(x_1,\ldots,x_{n-1})$ parametrize $\Sigma_\ast$ and such that the metric 
tensors fulfill
\begin{align}
 \left(g^i\right)_{nn} = 1 \,, \quad 
 \left(g^i\right)_{jn} = 0 \; \mbox{ for } \; j=1,\ldots,n-1 \,. 
 \label{localcoord}
\end{align}
This is possible by choosing the $n$'th coordinate as the distance from 
the $(n-1)$-dimensional surface $\Sigma_\ast$. 

Denoting the representation of the $u^j$, $j=1,2,3$, in local coordinates 
as $\hat{u}^j$, $j=1,2,3$, the principal parts of the boundary operators 
in \eqref{L_eq_1}  can be written as 
\begin{align*}
 \sum_{j=1}^3 \mathcal{B}_0^{ij} \hat{u}_j = 
  \left\{ \begin{array}{ll}
            \gamma^1 \hat{u}^1 + \gamma^2 \hat{u}^2 + \gamma^3 \hat{u}^3\,,
            & i=1 \,, \\
            \partial_n \hat{u}^1 - \partial_n \hat{u}^2\,, & i=2 \,, \\
            \partial_n \hat{u}^2 - \partial_n \hat{u}^3\,, & i=3 \,.
           \end{array}
  \right. 
\end{align*}
The principal part of the parabolic differential operator takes the form
$$
 \mathcal{L}_0(\partial_t,\nabla) = \left(l_0^{ij}\right)_{i,j=1,2,3}
$$
with 
$$
 l_0^{ij} = \left\{\begin{array}{cl}
 0, & i \neq j, \\[0.15cm]%
 \dis \partial_t - \sum_{k,l=1}^{n} \beta^i g^{i,kl}  \partial_k \partial_l,
 & i=j.
\end{array}\right.
$$
For $\xi \in \mathbb{R}^n$ and $p \in \mathbb{C}$ with positive real part 
we now define 
\begin{align*}
 \boldsymbol{L} := \det \mathcal{L}_0(p,\boldsymbol{i} \xi) 
  = \prod_{i=1}^3 \Bigl( 
                   p + \sum_{k,l=1}^{n} \beta^i g^{i,kl} \xi_k \xi_l 
                  \Bigr)
\end{align*}
and 
\begin{align*}
 \widehat{\mathcal{L}}_0 = \bigl( \widehat{l}_0^{\;ij} \bigr)_{i,j=1,2,3} 
  = \boldsymbol{L} \left( \mathcal{L}_0 \right)^{-1} \,.
\end{align*}
\begin{lemma} \label{lem:LopShapCond}
The operators $( \widehat{\mathcal{L}}_0 , \mathcal{B}_0 )$ 
fulfill the Lopantinskii-Shapiro conditions.
\end{lemma}
\begin{proof}
For the coefficients of $\widehat{\mathcal{L}}_0$ we calculate
$$
\widehat{l}_0^{\;ij}=\left\{\begin{array}{cl}
  0, & i \neq j, \\[0.15cm]%
  \dis \prod_{\stackrel{j=1}{j \neq i}}^3 
    \Bigl( p + \sum_{k,l=1}^{n} \beta^j g^{j,kl} \xi_k \xi_l \Bigr)\,,
  & i=j.
\end{array}\right.
$$
We now set $\xi = \xi' + \tau e_n$ with $\xi'_n = 0$, $\tau \in \mathbb{R}$ 
and $e_n = (0,\ldots,0,1)$. Let $\tau^i(p,\xi')$, $i=1,2,3$, be those roots 
of $\boldsymbol{L}(p,\boldsymbol{i} (\xi' + \tau e_n))$, which have positive 
imaginary part. The fact that there are exact three roots with positive 
imaginary part follows from the fact that the system in the first line 
of \eqref{L_eq_1} is parabolic. Now we define
$$
  \widehat{p}^{\;i} := p + \sum_{k,l=1}^{n-1} \beta^i g^{i,kl} \xi_k \xi_l \,,
$$
where $p$ is assumed to have a positive real part. We choose polar coordinates
\begin{align*}
  \widehat{p}^{\;i} = |\widehat{p}^{\;i}| e^{\boldsymbol{i} \,\phi^i} \,.
\end{align*}
The fact that $p$ has positive real part and the fact that 
$\left(g^{i,kl}\right)_{k,l=1,\ldots,n-1}$ is positive definite imply 
that $\phi^i \in \bigl(-\pi/2,\pi/2\bigr)$. Hence we compute
\begin{align}
  \tau^i(p,\xi') = \sqrt{|\widehat{p}^{\;i}|} \sqrt{\frac{1}{\beta^i}} 
  e^{\boldsymbol{i} \frac{1}{2} (\phi^i + \pi)} \,. \label{equ:tau}
\end{align}
The Lopantinskii-Shapiro conditions now require that the rows of 
the matrix $\mathcal{B}_0 \widehat{\mathcal{L}}_0$ are linearly independent 
for all $p \in \mathbb{C}$ with $Re \, p > 0$ modulo the polynomial
\begin{align*}
  M^+(p,\xi',\tau) = \prod_{i=1}^3 \left\{ \tau - \tau^i(p,\xi') \right\} \,.
\end{align*}
This can only be true if 
\begin{align*}
  \sum_{i=1}^3 \omega^i B_0^{ij}(\tau) = 0 \quad 
    \mbox{mod } \; \tau - \tau^j(p,\xi') \,, \; j=1,2,3 
\end{align*}
has a nontrivial solution $(\omega^1,\omega^2,\omega^3)$, where 
$B_0^{ij}(\tau) = \mathcal{B}_0^{ij}(\tau e_n)$. 
Hence we need to decide whether the set of equations 
 \begin{align}
  \sum_{i=1}^3 \omega^i B_0^{ij}(\tau^i(p,\xi')) = 0 \label{equboundlincombi}
 \end{align}
has a nontrivial solution. Using the definition of the $B_0^{ij}$ we finally 
need to decide whether the determinant of the matrix 
\begin{align*}
  \left(
   \begin{array}{ccc}
     \gamma^1 & \tau^1 & 0 \\
     \gamma^2 & -\tau^2 & \tau^2 \\
     \gamma^3 & 0 & -\tau^3
   \end{array}
  \right)
\end{align*}
is singular or not. Here we abbreviated $\tau^i = \tau^i(p,\xi')$. 
The determinant is given as
\begin{align} \label{detbo}
  \gamma^1 \tau^2 \tau^3 + \gamma^2 \tau^1 \tau^3 + \gamma^3 \tau^1 \tau^2 \,.
\end{align} 
In polar coordinates the angle of $\tau^i \tau^j$ is given as 
$(\phi^i + \phi^j)/2 + \pi$.
Since $\phi^i$, $\phi^j \in \bigl(-\pi/2,\pi/2\bigr)$ 
we obtain that $\tau^i \tau^j$ has negative real part. Hence 
$\gamma^1 \tau^2 \tau^3 + \gamma^2 \tau^1 \tau^3 + \gamma^3 \tau^1 \tau^2$ 
is the sum of three summands which all have negative real part. Hence 
the determinant is non-zero and we have shown that the Lopantinskii-Shapiro 
conditions hold. 
\end{proof}

\begin{remark}\label{rem:LopShap}
In Latushkin, Pr\"uss and Schnaubelt \cite{LPS} the Lopatinskii-Shapiro 
condition is formulated as a condition for a system of ordinary differential 
equations. In our notation this reads as follows. Let $\sigma \in \Sigma_\ast$ 
and $(x_1,\ldots,x_n)$ be local coordinates (in a region $\Omega$) as in 
\eqref{localcoord} and set 
$\mathcal{A}_0(\nabla) = \operatorname{diag}((-\sum_{k,l=1}^n \beta^i g^{i,kl} 
\partial_k \partial_l)_{i=1,2,3})$. 
Then the formulation in \cite{LPS} requires that for given 
$\xi \in \mathbb{R}^n$ with $\xi \perp n$ and 
$\lambda \in \{z \in \mathbb{C} \,|\, Re(z) \geq 0\}$
with $(\lambda,\xi) \neq (0,0)$ the function $\varphi = 0$ is the only 
bounded solution in $C_0(\mathbb{R}_+;\mathbb{C}^3)$ of the ODE-system 
\begin{align}
  &\lambda \varphi(y) + \mathcal{A}_0(i \xi + n(x) \partial_y) \varphi(y) 
   = 0 \,, \quad y>0 \,, \label{equivLS1} \\
  &\mathcal{B}_0(i\xi + n(x) \partial_y) \varphi(0) = 0 \,. \label{equivLS2}
\end{align}
The equivalence of the formulation in \cite{LPS} to the algebraic formulation 
in Solonnikov \cite{Sol} can be found in Eidelman and Zhitarashu 
\cite[Chap. I.2]{EZ98}.

By choosing for simplicity as above $\xi = (\xi',0)$ and $n(x) = e_n$ 
the equations~\eqref{equivLS1} and~\eqref{equivLS2} reduces in our case to
\begin{alignat}{2}
  &\lambda \varphi^j + |\xi'|^2 \beta^j \varphi^j - \beta^j (\varphi^j)'' = 0, \quad 
  && y>0 \,, \label{equ:easy1} \\
  &\gamma^1 \varphi^1 + \gamma^2 \varphi^2 + \gamma^3 \varphi^3 = 0,
  && y=0 \,, \label{equ:easy2} \\
  &(\varphi^1)' = (\varphi^2)' = (\varphi^3)',
  && y=0 \,. \label{equ:easy3}
\end{alignat}              
These equations can be treated with an energy method to show that 
a solution must be zero. To this end we test line \eqref{equ:easy1} with
$\gamma^j \overline{\varphi}^j/\beta^j$ and sum over $j=1,2,3$ to get
 \begin{align*}
  0 &= \sum_{j=1}^3 (\lambda + \beta^j |\xi'|^2) \frac{\gamma^j}{\beta^j} 
       \int_0^\infty |\varphi^j|^2 \,dy 
       - \sum_{j=1}^3 \gamma^j \int_0^\infty (\varphi^j)'' \, \overline{\varphi}^j \,dy \\
    &= \sum_{j=1}^3 (\lambda + \beta^j |\xi'|^2) \frac{\gamma^j}{\beta^j} 
       \int_0^\infty |\varphi^j|^2 \,dy 
       + \sum_{j=1}^3 \gamma^j \int_0^\infty |(\varphi^j)'|^2 \,dy 
       - \sum_{j=1}^3 \gamma^j (\varphi^j)'(0) \, \overline{\varphi}^j(0) \\
    &= \sum_{j=1}^3 (\lambda + \beta^j |\xi'|^2) \frac{\gamma^j}{\beta^j} 
       \int_0^\infty |\varphi^j|^2 \,dy 
       + \sum_{j=1}^3 \gamma^j \int_0^\infty |(\varphi^j)'|^2 \,dy 
       - (\varphi^1)'(0) \sum_{j=1}^3 \gamma^j \overline{\varphi}^j(0) \,.
 \end{align*}
In the last line we used the boundary condition \eqref{equ:easy3}. 
Finally with \eqref{equ:easy2} we see that the last term vanishes and 
that therefore $(\varphi^1,\varphi^2,\varphi^3)=0$.
\end{remark}



\noindent
{\it Proof of Theorem \ref{thm:L_exist}.} 
First we construct a weak solution of problem \eqref{L_eq_1} without 
the nonlocal term. 
In order to apply an energy method we modify the equations into
\begin{equation}
\left\{\begin{array}{ll} \label{L_sys}
\dfrac{\gamma^i}{\beta^i} \partial_t u^i 
= \gamma^i (\Delta_{\G^i_\ast} u^i 
  + |\Pi^i_\ast|^2 u^i) + \dfrac{\gamma^i}{\beta^i} f^i  
& \mbox{on } \, \G^i_\ast \times[0,T] \,, \\
\dis\sum_{j=1}^3{\cal B}^{ij}u^j = b^i 
& \mbox{on } \, \Sigma_\ast\times[0,T] \,, \\ [0.6cm]
u^i\big|_{t=0} = 0 
& \mbox{on } \, \G^i_\ast \,.
\end{array}\right.
\end{equation}
In this way we are able to choose the weak solution $\bmu=(u^1,u^2,u^3)$ and 
the test functions $\boldsymbol{\xi} = (\xi^1,\xi^2,\xi^3)$ in the same space. 
Now we introduce the function spaces
\begin{align*}
&{\cal L} 
:= L^2(\Gamma_\ast^1) \times L^2(\Gamma_\ast^2) \times 
L^2(\Gamma_\ast^3), \quad
{\cal L}_b 
:=L^2(\partial\Gamma_\ast^1) \times L^2(\partial\Gamma_\ast^2) 
\times L^2(\partial\Gamma_\ast^3) \;(= \left(L^2(\Sigma_\ast)\right)^3 ), \\
&{\cal H}^1 
:= H^1(\Gamma_\ast^1) \times H^1(\Gamma_\ast^2) \times 
H^1(\Gamma_\ast^3), \quad
{\cal E}
:=\{\bmu\in{\cal H}^1 \,|\, \gamma^1 u^1+\gamma^2 u^2+\gamma^3 u^3 = 0 \, 
\mbox{ a.e. on } \, \Sigma_\ast \}.
\end{align*}
Also, we introduce the time-dependent bilinear form 
\begin{align*}
B[\bmu,\bmxi;t]
:=&\,\sum_{i=1}^3\gamma^i\biggl\{ 
\int_{\Gamma_\ast^i}\la\nabla_{\Gamma_\ast^i}u^i,\nabla_{\Gamma_\ast^i}\xi^i\ra
\,d\mathcal{H}^{n}
- \int_{\Gamma_\ast^i}|\Pi_\ast^i|^2u^i\xi^i \,d\mathcal{H}^{n} 
+ \int_{\Sigma_\ast} \kappa_\ast^i \left( \mathcal{T} u \right)^i \xi^i 
\,d\mathcal{H}^{n-1} \biggr\}
\end{align*}
for $\bmu(\cdot,t),\bmxi(\cdot,t)\in{\cal E}$.
The weak formulation then reads as follows. 
Find $\bmu \in L^2(0,T;\mathcal{E})$ with 
$\partial_t \bmu \in L^2(0,T;\left(\mathcal{H}^1\right)^{-1})$ such that
\begin{align} \label{weakform}
 \la \partial_t \bmu , \boldsymbol{\xi} \ra_{dual} 
 + B[\bmu,\boldsymbol{\xi};t] 
 = \left( \boldsymbol{f} , \boldsymbol{\xi} \right)_{\mathcal{L}} 
   + b(\boldsymbol{\xi};t)
 & \quad \mbox{for all}\,\ \boldsymbol{\xi} \in \mathcal{E}\,\ 
   \mbox{and a.e. in $t$} ,
\end{align}
where $\left(\mathcal{H}^1\right)^{-1}$ is the dual space to $\mathcal{H}^1$ 
and 
\begin{equation}\label{norm_L2}
\la \partial_t \bmu , \boldsymbol{\xi} \ra_{dual} 
= \sum_{i=1}^3 \frac{\gamma^i}{\beta^i} \la \partial_t u^i , \xi^i \ra_{dual}, 
\quad
(\boldsymbol{f} , \boldsymbol{\xi})_{\mathcal{L}} 
= \sum_{i=1}^3 \frac{\gamma^i}{\beta^i} (f^i , \xi^i)_{L^2(\G^i_\ast)}
\end{equation}
are scaled versions of the corresponding duality pairing and 
 inner product. The time-dependent linear form $b$ is given through 
$$
b(\boldsymbol{\xi};t) = \int_{\Sigma_\ast} \left(\gamma^1 (b^2 + b^3) \xi^1 
+ \gamma^2 b^3 \xi^2 \right)\,d\mathcal{H}^{n-1}
$$
and consists of terms which appear formally due to the rewriting of 
$\int_{\Sigma_\ast}\gamma^i\partial_{\nu^i_\ast} u^i \xi^i
\,d\mathcal{H}^{n-1}$ to make use of $\sum_{i=1}^3 \gamma^i \xi^i = 0$.
That this weak formulation for smooth solutions is equivalent to the strong 
formulation, can be checked by a straightforward computation using integration 
by parts and the restriction $\boldsymbol{\xi} \in \mathcal{E}$.

We want to apply the Galerkin method and therefore assume that 
$\bmw_k=\bmw_k(\s)$ for $k=1,2,\ldots$ are smooth functions such that 
$\{\bmw_k\}_{k=1}^{\infty}$ is an orthonormal basis in ${\cal L}$. 
Indeed, we can take such $\{\bmw_k\}_{k=1}^{\infty}$ considering 
eigenfunctions of the eigenvalue problem
\begin{align*}
\left\{ \begin{array}{ll}
-\gamma^i \Delta_{\G^i_\ast} w^i 
= \lambda\,\dfrac{\gamma^i}{\beta^i} w^i 
  & \mbox{on } \; \G_\ast^i \,, \; i=1,2,3 \,, \\
\gamma^1 w^1+\gamma^2 w^2+\gamma^3 w^3 = 0 
  & \mbox{on } \; \Sigma_\ast \,, \\
\la\nabla_{\Gamma_\ast^1}w^1,\nu_\ast^1\ra  
= \la\nabla_{\Gamma_\ast^2}w^2,\nu_\ast^2\ra
= \la\nabla_{\Gamma_\ast^3}w^3,\nu_\ast^3\ra 
  & \mbox{on }  \; \Sigma_\ast \,.
\end{array}\right.
\end{align*}
This follows similar as in Gilbarg and Trudinger \cite{GT} by considering 
the quadratic form
\begin{align*}
 Q(\bmu,\bmu) = \sum_{i=1}^3 \gamma^i 
 \int_{\G^i_\ast} \la\nabla_{\G^i_\ast} u^i, \nabla_{\G^i_\ast} u^i\ra 
 \, d\mathcal{H}^{n}
\end{align*}
on $\mathcal{E}$ and the norm of ${\cal L}$ as in (\ref{norm_L2}). 
In addition the eigenfunctions are orthogonal with respect 
to the quadratic form $Q$. 
We remark, that since the boundary conditions fulfill the Lopantinskii-Shapiro conditions, 
one can also derive regularity results for the eigenfunctions 
$\{\bmw_k\}_{k=1}^{\infty}$. 

Now fix a positive integer $m \in \mathbb{N}$ and look for 
$\bmu_m : [0,T] \to {\cal E}$ of the form
\begin{equation} \label{series}
\bmu_m(t) = \sum_{k=1}^md_m^k(t)\bmw_k.
\end{equation}
Here the coefficients $d_m^k(t)$ for $k=1,2,\ldots,m$ have to be chosen such that 
\begin{align}
 d_m^k(0) 
 &= 0 \,, \label{d_1} \\
 (\partial_t\bmu_m,\bmw_k)_{{\cal L}}+B[\bmu_m,\bmw_k;t] 
 &= (\boldsymbol{f} ,\bmw_k)_{{\cal L}} +  b(\bmw_k;t) \,, \label{d_2}
\end{align}
where $k=1,\ldots,m$ and the second line has to be understood pointwise 
in $t$. 
Note that due to $\bmw_k \in \mathcal{E}$ a function $\bmu_m$ of 
the form \eqref{series} satisfies
$$
\gamma^1u_m^1(\s,t) + \gamma^2u_m^2(\s,t) + \gamma^3u_m^3(\s,t)=0
\quad \mbox{for } \, \s \in \Sigma_\ast \,.
$$
With the help of theory for linear systems of ordinary differential equations 
we find $(d_m^1,\ldots,d_m^m)$ as a unique solution of 
$$
(d_m^k)'(t)+\sum_{l=1}^mB[\bmw_l,\bmw_k;t]d_m^l(t)
=(\boldsymbol{f}(\cdot\,,t),\bmw_k)_{{\cal L}} + b(\bmw_k;t)
$$
with the initial data (\ref{d_1}), so that $\bmu_m$ of the form 
(\ref{series}) satisfies (\ref{d_1}) and (\ref{d_2}) for each 
$m\in{\mathbb N}$. 

Since the trace operator is compact one can use a contradiction 
argument similar as in the proof of the Ehrling Lemma in order 
to derive the inequality
\begin{align*}
\|\bmu\|_{{\cal L}_b}^2
\le\epsi\|\nabla\bmu\|_{{\cal L}}^2+C_\epsi\|\bmu\|_{{\cal L}}^2
\quad (\bmu\in{\cal E})
\end{align*}
for each $\epsi>0$ and a constant $C_\epsi>0$. 
Using this inequality one can argue similar as in the proof of 
Evans \cite[Sect. 7.1.2, Th. 2]{E} and obtain the energy estimate 
\begin{equation}
\sup_{0\le t\le T}\|\bmu_m(t)\|_{{\cal L}}
+\|\bmu_m\|_{L^2(0,T;{\cal E})}
+\|\partial_t\bmu_m\|_{L^2(0,T;\left(\mathcal{H}^1\right)^{-1})}
\le C(\|\bmf\|_{L^2(0,T;{\cal L})})
+\|\bmb\|_{L^2(0,T;{\cal L}_b)})
\label{energy_es}
\end{equation}
for $m\in{\mathbb N}$ and a constant $C>0$. Using this we can prove 
the existence and uniqueness of a weak solution with standard arguments, 
which can be found for example in Evans \cite[p.356--358]{E}.

Let us derive Schauder estimates for solutions of problem \eqref{L_sys}. 
Here we consider the H\"older estimate only near the triple junction and 
just remark that away from the triple junction the result follows 
in a standard way after localization.

Let us introduce some notation. Locally around a point $\sigma \in \Sigma_\ast$ we choose parametrizations which flatten the boundary in the following way. 
We pick a sequence $0 < r_1 < r_2 < r_3 < r_4$ and with 
$Q_l := B_{r_l}(y) \cap \{x \in \mathbb{R}^n \,|\, x_n \geq 0\}$
for $l=1,2,3,4$, where $y \in \mathbb{R}^n$ is such that $y_n=0$, 
we let $F^i:Q_4 \to \G^i_\ast$, $i=1,2,3$, be local parametrizations with 
$F^i(y) = \sigma$ and $\left. F^i \right|_{\{x_n=0\}} \subset \Sigma_\ast$. 
Additionally for a given $t_0 \geq 0$ we choose a sequence 
$0 < \delta_1 < \delta_2 < \delta_3 < \delta_4$ and set 
$\Lambda_l := (t_0 - \delta_l , t_0 + \delta_l) \cap 
\{t \in \mathbb{R} \,|\, t \geq 0\}$ for $l=1,2,3,4$. 

With the help of a cut-off function we will formulate problem \eqref{L_sys} 
for the representations $\hat{u}^j = u^j \circ F^j$ in $Q_4 \times \Lambda_4$ 
in Euclidean space. To preserve the structure of the problem and to keep 
the notation simple, we will identify the notation of the function $u^j$ with 
its representation in local coordinates. In the next steps the sets $Q_l \times \Lambda_l$ 
will be successively reduced to achieve finally the stated H\"older estimate 
in $Q_1 \times \Lambda_1$. We will need the following notation for parts 
of the boundary of $Q_l$:
\begin{align*}
C_l := \partial Q_l \cap \{x \in \mathbb{R}^n \,|\, x_n > 0 \} 
\quad \mbox{and} \quad 
S_l := \partial Q_l \backslash C_l \,.
\end{align*}
Now let $\eta$ be a cut-off function satisfying
\begin{align*}
\eta\in C^{\infty}_0(Q_4 \times \Lambda_4) \,, \quad 0 \le \eta \le 1, \quad 
\eta \equiv 1 \,\ \mbox{on}\,\ Q_3 \times \Lambda_3 \,.
\end{align*}
We remark that due to the fact that $Q_4$ is not open, the values $\eta(x,t)$ 
for $x \in S_4$ do not necessarily vanish. The same holds true for 
$\eta(x,0)$, if $\Lambda_4$ is not open. 

Now set $v^i=\eta u^i$, where $(u^1,u^2,u^3)$ is a weak solution of 
\eqref{L_sys} and note that we do not distinguish between the functions 
$u^i$ and its representations. Then we have in a weak sense
\begin{align*}
& \partial_t v^i 
= \eta \partial_t u^i + \partial_t \eta \,u^i, \quad 
\Delta_{\Gamma_\ast^i}v^i 
= \eta\Delta_{\Gamma_\ast^i}u^i 
  + 2 \la\nabla_{\Gamma_\ast^i}\eta,\nabla_{\Gamma_\ast^i}u^i\ra
  +(\Delta_{\Gamma_\ast^i}\eta)u^i, \\
& \la\nabla_{\Gamma_\ast^i}v^i,\nu_\ast^i\ra 
= \eta\la\nabla_{\Gamma_\ast^i}u^i,\nu_\ast^i\ra
  +\la\nabla_{\Gamma_\ast^i}\eta,\nu_\ast^i\ra u^i.
\end{align*}
Since $(u^1,u^2,u^3)$ is a weak solution of \eqref{L_sys}, we deduce that 
$(v^1,v^2,v^3)$ is a weak solution of
\begin{equation} \label{L_sys_v}
 \left\{ 
  \begin{array}{ll}
   \dfrac{\gamma^i}{\beta^i} \partial_t v^i
   =\gamma^i(\Delta_{\Gamma_\ast^i}v^i+|\Pi_\ast^i|^2 v^i)+\tilde{f}^i(x,t), 
     & (x,t) \in Q_4 \times \Lambda_4 ,   \\[0.15cm]
   \gamma^1 v^1+\gamma^2 v^2+\gamma^3 v^3 = 0,
     & (x,t)\in S_4 \times \Lambda_4, \\[0.1cm]
   \la\nabla_{\Gamma_\ast^1}v^1,\nu_\ast^1\ra 
   + \kappa_\ast^1 (\mathcal{T} \bmv)^1 
   - \la\nabla_{\Gamma_\ast^2}v^2,\nu_\ast^2\ra 
   - \kappa_\ast^2 (\mathcal{T} \bmv)^2 
   =\tilde{b}^2(x,t), 
     & (x,t)\in S_4 \times \Lambda_4, \\[0.25cm]
   \la\nabla_{\Gamma_\ast^2}v^2,\nu_\ast^2\ra 
   + \kappa_\ast^2 (\mathcal{T} \bmv)^2 
   - \la\nabla_{\Gamma_\ast^3}v^3,\nu_\ast^3\ra 
   - \kappa_\ast^3 (\mathcal{T} \bmv)^3 
   =\tilde{b}^3(x,t), 
     & (x,t)\in S_4 \times \Lambda_4, \\[0.25cm]
   v^i(x,t) = 0 ,  & (x,t)\in C_4 \times \Lambda_4, \\[0.15cm]
   v^i(x,0) = 0 ,  &  x\in Q_4 ,
  \end{array}
 \right.
\end{equation}
where $i=1,2,3$ and 
\begin{align*}
\tilde{f}^i 
 &= \dfrac{\gamma^i}{\beta^i} \eta^i f^i 
  + \dfrac{\gamma^i}{\beta^i} \partial_t \eta\, u^i
  -2\gamma^i\la\nabla_{\Gamma_\ast^i}\eta,\nabla_{\Gamma_\ast^i}u^i\ra
  -\gamma^i(\Delta_{\Gamma_\ast^i}\eta)u^i, \\
 \tilde{b}^2 
 &= \eta^i b^2 + \la\nabla_{\Gamma_\ast^1}\eta,\nu_\ast^1\ra u^1
  - \la\nabla_{\Gamma_\ast^2}\eta,\nu_\ast^2\ra u^2, \\
 \tilde{b}^3 
 &= \eta^i b^3 + \la\nabla_{\Gamma_\ast^2}\eta,\nu_\ast^2\ra u^2
  - \la\nabla_{\Gamma_\ast^3}\eta,\nu_\ast^3\ra u^3.
\end{align*}
Note that 
\begin{align*}
&\tilde{f}^i|_{Q_3 \times \Lambda_3}
   = f^i\in C^{\alpha,\frac{\alpha}2}(Q_3\times\Lambda_3), \quad
 \tilde{f}^i|_{Q_4 \times \Lambda_4} \in L^2(Q_4 \times \Lambda_4), \\
&\tilde{b}^i|_{S_3 \times \Lambda_3}
= b^i \in C^{1+\alpha,\frac{1+\alpha}2}(S_3 \times \Lambda_3), \quad
\tilde{b}^i|_{S_4 \times \Lambda_4} \in L^2(S_4 \times \Lambda_4).
\end{align*}
Let $\tilde{f}^i_n$ and $\tilde{b}^i_n$ be smooth approximations of 
$\tilde{f}^i$ and $\tilde{b}^i$ satisfying
\begin{equation}
\|\tilde{f}^i_n-\tilde{f}^i\|_{L^2(Q_4 \times \Lambda_4)} \to 0, \quad
\|\tilde{b}^i_n-\tilde{b}^i\|_{L^2(S_4 \times \Lambda_4)} \to 0
\label{L2-approx}
\end{equation}
and on $Q_2 \times \Lambda_2 \subset Q_3 \times \Lambda_3$ we require 
\begin{align*}
&\|\tilde{f}^i_n\|_{C^{\alpha,\frac{\alpha}2}(Q_2 \times \Lambda_2)}
\le\|\tilde{f}^i\|_{C^{\alpha,\frac{\alpha}2}(Q_3 \times \Lambda_3)}
=\|f^i\|_{C^{\alpha,\frac{\alpha}2}(Q_3 \times \Lambda_3)}, \\
&\|\tilde{b}^i_n\|_{C^{1+\alpha,\frac{1+\alpha}2}(S_2\times\Lambda_2)}
\le\|\tilde{b}^i\|_{C^{1+\alpha,\frac{1+\alpha}2}(S_3\times\Lambda_3)}
=\|b^i\|_{C^{1+\alpha,\frac{1+\alpha}2}(S_3\times\Lambda_3)}.
\end{align*}
Replace $\tilde{f}^i$ and $\tilde{b}^i$ by $\tilde{f}^i_n$ and 
$\tilde{b}^i_n$ in \eqref{L_sys_v}, and call this problem 
\eqref{L_sys_v}$_n$. Since we checked the Lopatinskii-Shapiro conditions 
on the triple junction in Lemma \ref{lem:LopShapCond}, we can apply results 
from Solonnikov \cite[Theorem 4.9]{Sol} to get a unique solution 
$v^i_n\in C^{2+\alpha,1+\frac{\alpha}{2}}(Q_4\times\Lambda_4)$ of problem 
\eqref{L_sys_v}$_n$. Using furthermore the local estimate from 
\cite[Theorem 4.11]{Sol}, we obtain for 
$Q_1\times\Lambda_1 \subset Q_2\times\Lambda_2 \subset Q_3\times\Lambda_3$
from above that
\begin{align}
\begin{split}
\sum_{i=1}^3\|v^i_n\|_{C^{2+\alpha,1+\frac{\alpha}2}(Q_1\times\Lambda_1)}
\le {}& \, C_1\biggl\{\sum_{i=1}^3
\|\tilde{f}^i_n\|_{C^{\alpha,\frac{\alpha}2}(Q_2\times\Lambda_2)}
+\sum_{i=2}^3
\|\tilde{b}^i_n\|_{C^{1+\alpha,\frac{1+\alpha}2}(S_2\times\Lambda_2)}
\biggr\} \\
{}&+C_2\sum_{i=1}^3\|v^i_n\|_{L^2(Q_2\times\Lambda_2)} \\
\le {}&\,C_1\biggl\{\sum_{i=1}^3
\|f^i\|_{C^{\alpha,\frac{\alpha}2}(Q_3\times\Lambda_3)}
+\sum_{i=2}^3
\|b^i\|_{C^{1+\alpha,\frac{1+\alpha}2}(S_3\times\Lambda_3)}
\biggr\} \\
 {}&+C_2\sum_{i=1}^3\|v^i_n\|_{L^2(Q_4\times\Lambda_4)}.
\label{est_1}
\end{split}
\end{align}
By means of (\ref{L2-approx}) and the energy estimate \eqref{energy_es} for 
the approximated problem \eqref{L_sys_v}$_n$, we see \allowdisplaybreaks
\begin{align}
\begin{split}
\sum_{i=1}^3\|v^i_n\|_{L^2(\Lambda_4,H^1(Q_4))}
\le {}&\,C\biggl\{
\sum_{i=1}^3\|\tilde{f}^i_n\|_{L^2(Q_4\times\Lambda_4)}
+\sum_{i=2}^3\|\tilde{b}^i_n\|_{L^2(S_4\times\Lambda_4)}
\biggr\} \\
\le {}&\,\tilde{C}\biggl\{
\sum_{i=1}^3\|\tilde{f}^i\|_{L^2(Q_4\times\Lambda_4)}
+\sum_{i=2}^3\|\tilde{b}^i\|_{L^2(S_4\times\Lambda_4)}
\biggr\} \\
\le {}&\,\tilde{C}\biggl\{
\sum_{i=1}^3\|f^i\|_{L^2(Q_4\times\Lambda_4)}
+\sum_{i=2}^3\|b^i\|_{L^2(S_4\times\Lambda_4)}  \\
&\qquad +\sum_{i=1}^3\|u^i\|_{L^2(\Lambda_4,H^1(Q_4))}
\biggr\}  \\ 
\le {}&\,\tilde{C}'\biggl\{
\sum_{i=1}^3\|f^i\|_{L^2(Q_4\times\Lambda_4)}
+\sum_{i=2}^3\|b^i\|_{L^2(S_4\times\Lambda_4)}
\biggr\}. \label{est_2}
\end{split}
\end{align}
In the last inequality we used the energy estimate \eqref{energy_es}.
From the last bound we deduce the existence of a subsequence
$\{v^i_{n_\ell}\}\subset\{v^i_n\}$ and of 
$\bar{v}^i\in L^2(\Lambda_4,H^1(Q_4))$ such that 
$$
v^i_{n_\ell} \to \bar{v}^i,\,\ \mbox{weakly},
$$
and $(\bar{v}^1,\bar{v}^2,\bar{v}^3)$ is a weak solution of \eqref{L_sys_v}. 
By uniqueness of the weak solution of \eqref{L_sys_v}, 
$$
\bar{v}^i=v^i\,\ \mbox{in}\,\ Q_4\times\Lambda_4.
$$
Let us rewrite $v^i_{n_\ell}$ as $v^i_\ell$. By \eqref{est_1} and 
\eqref{est_2}, we obtain
\begin{align*}
\sum_{i=1}^3
\|v^i_\ell\|_{C^{2+\alpha,1+\frac{\alpha}2}(Q_1\times\Lambda_1)}
\le {}&\,C_1\biggl\{\sum_{i=1}^3
\|f^i\|_{C^{\alpha,\frac{\alpha}2}(Q_3\times\Lambda_3)}
+\sum_{i=2}^3
\|b^i\|_{C^{1+\alpha,\frac{1+\alpha}2}(S_3\times\Lambda_3)}
\biggr\} \\
 {}& +\tilde{C}_2\biggl\{
\sum_{i=1}^3\|f^i\|_{L^2(Q_4\times\Lambda_4)}
+\sum_{i=2}^3\|b^i\|_{L^2(S_4\times\Lambda_4)}
\biggr\} \\
\le {}&\,C\biggl\{\sum_{i=1}^3
\|f^i\|_{C^{\alpha,\frac{\alpha}2}(Q_4\times\Lambda_4)}
+\sum_{i=2}^3
\|b^i\|_{C^{1+\alpha,\frac{1+\alpha}2}(S_4\times\Lambda_4)}
\biggr\}.
\end{align*}
Then, by the theorem of Arzel\`{a}-Ascoli, there exist 
$\{v^i_{\ell_m}\}\subset\{v^i_\ell\}$ and 
$\hat{v}^i\in C^{2,1}(Q_1\times\Lambda_1)$ such that 
$$
v^i_{\ell_m}\to \hat{v}^i\,\ \mbox{in}\,\ C^{2,1}(Q_1\times\Lambda_1).
$$
Here $\hat{v}^i$ is in 
$C^{2+\alpha,1+\frac{\alpha}2}(Q_1\times\Lambda_1)$ 
because of, for example, 
$$
|\nabla_j\nabla_k\hat{v}^i(x)-\nabla_j\nabla_k\hat{v}^i(y)|
=\lim_{m\to\infty}
|\nabla_j\nabla_kv^i_{\ell_m}(x)-\nabla_j\nabla_kv^i_{\ell_m}(y)|
\le C|x-y|^{\alpha}.
$$
It follows from uniqueness of a limit and $\bar{v}^i=v^i$ in 
$Q_4\times\Lambda_4$ that 
$$
\hat{v}^i=\bar{v}^i=v^i\,\ \mbox{in}\,\ Q_1\times\Lambda_1.
$$
Since $v^i=u^i$ in $Q_1\times\Lambda_1$, $u^i$ is in 
$C^{2+\alpha,1+\frac{\alpha}2}(Q_1\times\Lambda_1)$ and satisfies
\begin{align*}
\sum_{i=1}^3
\|u^i\|_{C^{2+\alpha,1+\frac{\alpha}2}(Q_1\times\Lambda_1)}
\le {}&\,C\biggl\{\sum_{i=1}^3
\|f^i\|_{C^{\alpha,\frac{\alpha}2}(Q_4\times\Lambda_4)}
+\sum_{i=2}^3
\|b^i\|_{C^{1+\alpha,\frac{1+\alpha}2}(S_4\times\Lambda_4)}
\biggr\}.
\end{align*}
Hence we are led to the stated H\"older estimate locally around 
the triple junction $\Sigma_\ast$. By a covering argument we can enlarge 
the estimate to a neighbourhood of $\Sigma_\ast$ and then 
by an easier argument, that we omit here, we can give it for 
all hypersurfaces $\G^i_\ast$ as claimed. 

Finally, by a perturbation argument as in Baconneau and Lunardi 
\cite[Thm. 2.3]{BL}, we derive the existence of a unique solution 
and the Schauder estimate for the linearized system with nonlocal term. 
We omit the details since this part is even easier than in \cite{BL} 
due to the fact that the nonlocal terms $(\mathcal{T} (\bmu \circ \pr^i))^i$ 
do not contain derivatives of $\bmu$.

Altogether we proved Theorem \ref{thm:L_exist}.
\qed

\begin{remark} \label{rem:initialdate}
 For the case of arbitrary initial date $u^i\big|_{t=0} = \rho^i_0$, 
we have the following existence result. Let $\alpha\in(0,1)$. 
Then there exists $\delta_0>0$ such that for every 
$f^i\in C^{\alpha,\frac{\alpha}2}(Q^i_{\delta_0})$, 
$b^i\in C^{1+\alpha,\frac{1+\alpha}2}(\Sigma_\ast\times[0,\delta_0])$ 
with $b^1 \equiv 0$ and $\rho^i_0 \in C^{2+\alpha}(\G^i_\ast)$ with 
the compatibility condition
\begin{align*}
  (\gamma^1 f^1+\gamma^2 f^2+\gamma^3 f^3)\big|_{t=0} 
  = -\sum_{i=1}^3 \gamma^i \left( \mathcal{A}^i \rho^i_0 
    + \zeta^i \left( \mathcal{T} \brho_0 \right)^i \right),\,\,\ 
  b^i\big|_{t=0} = \sum_{j=1}^3 \mathcal{B}^{ij} \rho^j_0  \,\,\ 
    \mbox{on}\,\ \Sigma_\ast,
\end{align*}
the problem
\begin{equation}
  \left\{\begin{array}{ll}
          u^i_t={\cal A}^iu^i+\zeta^i({\cal T}(\bmu\circ\pr^i))^i+f^i
           & \mbox{on } \; \G^i_\ast\times[0,T], \\
          \dis\sum_{j=1}^3{\cal B}^{ij}u^j=b^i 
           & \mbox{on } \; \Sigma_\ast\times[0,T], \\[0.6cm]
          u^i\big|_{t=0}=\rho^i_0 & \mbox{on } \; \G^i_\ast
         \end{array}
   \right.
  \label{L_eq_1withinitdate}
\end{equation}
for $i=1,2,3$ has a unique solution $(u^1,u^2,u^3)\in{\cal X}_{\delta_0}$. 
Moreover, there exists $C>0$, which is independent of $\delta_0$, such that
$$
 \sum_{i=1}^3\|u^i\|_{C^{2+\alpha,1+\frac{\alpha}2}(Q^i_{\delta_0})}
 \le C\sum_{i=1}^3\bigl\{\|f^i\|_{C^{\alpha,\frac{\alpha}2}(Q^i_{\delta_0})}
 +\|g^i\|_{C^{1+\alpha,\frac{1+\alpha}2}(\Sigma_\ast\times[0,\delta_0])}
 +\|\rho^i_0\|_{C^{2+\alpha}(Q^i_{\delta_0})} \bigr\}.
$$
For the proof consider the difference $v^i := u^i - \rho^i_0$ and apply 
Theorem \ref{thm:L_exist} to $v^i$.
\end{remark}

\section{Local existence} \label{sec:locexist}

With the help of the previous results we are now in a position to solve 
the nonlinear nonlocal problem \eqref{L_eq} locally in time. 
We will apply a method similar to Lunardi \cite[Th. 8.5.4]{Lu} resp. 
Baconneau and Lunardi \cite{BL}. But since we do not linearize around 
the initial state and since our problem is geometrically more involved, we state some of the arguments in detail. Note that 
for $T>0$ and $0<\alpha<1$ we use the H\"older spaces
\begin{align*}
\mathcal{X}_T 
= C^{2+\alpha,1+\frac{\alpha}{2}}(Q^1_T) \times 
  C^{2+\alpha,1+\frac{\alpha}{2}}(Q^2_T) \times 
  C^{2+\alpha,1+\frac{\alpha}{2}}(Q^3_T) \,,
\end{align*}
where $Q^i_T = \G^i_\ast \times [0,T]$. Roughly we show in the following 
theorem that if the initial state satisfies the compatibility conditions 
and lies $C^{2+\alpha}$-close to the reference state,
there is a unique solution $(u^1,u^2,u^3) \in \mathcal{X}_\delta$ of 
\eqref{L_eq} where $\delta > 0$ is chosen sufficiently small.

\begin{theo} \label{bigexisttheorem}
Assume that $\rho^i_0 \in C^{2+\alpha}(\Gamma^i_\ast)$, $i=1,2,3$, fulfill 
the compatibility conditions \eqref{comp:rho_0}. 
Then there exist constants $R_0>0$ and $\varepsilon_0>0$ such that 
for each $R\ge R_0$ there is $\delta>0$ satisfying that if 
$\sum_{i=1}^3 \|\rho^i_0\|_{C^{2+\alpha}(\Gamma_\ast^i)} \leq \varepsilon_0$, 
the nonlinear nonlocal problem \eqref{L_eq} has a unique solution 
$u=(u^1,u^2,u^3)$ in the ball $B_R(\brho_0) \subset \mathcal{X}_\delta$. 
\end{theo}
 
\begin{proof}
Let $r>0$ be a constant such that for $v^i\in C^2(\Gamma_\ast^i)$ with 
$\sum_{i=1}^3\|v^i\|_{C^2(\Gamma_\ast^i)}\leq r$ the following assumptions hold:
\begin{enumerate}
\item[(A1)] 
${\cal F}^i(v^i,\left. \bmv \right|_{\Sigma_\ast})$ and 
$\mathfrak{a}^i_{\dag}(v^i,\left. \bmv \right|_{\Sigma_\ast})$ 
(see (\ref{rho_equ_rewritten})) are well-defined as well as 
${\cal P}(\bmv,\left. \bmv \right|_{\Sigma_\ast})$ (see (\ref{mathcal_P})).
\item[(A2)] 
Any first order derivatives of ${\cal F}^i$ with respect to 
$v^i$, $\left. \bmv \right|_{\Sigma_\ast}$, $\nabla_j v^i$, 
$\bar{\nabla}_j \left. \bmv \right|_{\Sigma_\ast}$, 
$\nabla_j \nabla_k v^i$ and $\bar{\nabla}_j \bar{\nabla}_k 
\left. \bmv \right|_{\Sigma_\ast}$ 
are locally Lipschitz continuous with respect to those. 
Also, any first order derivatives of $\mathfrak{a}^i_{\dag}$ with respect to 
$v^i$, $\left. \bmv \right|_{\Sigma_\ast}$, $\nabla_j v^i$ and 
$\bar{\nabla}_j \left. \bmv \right|_{\Sigma_\ast}$ are locally Lipschitz 
continuous with respect to those.
\item[(A3)] 
Any second order derivatives of $\mathfrak{b}^i$ with respect to 
$\left. \bmv \right|_{\Sigma_\ast}$ and 
$\bar{\nabla}_j \left. \bmv \right|_{\Sigma_\ast}$ are 
locally Lipschitz continuous with respect to those.
\end{enumerate}

We remark that these properties are realized for sufficiently small $r$
since with the notations 
$z^i_1 =(v^i,\nabla v^i,\left. \bmv \right|_{\Sigma_\ast},
\bar{\nabla} \left. \bmv \right|_{\Sigma_\ast} )$ 
and $z^i_2 =(v^i, \nabla v^i, \nabla^2 v^i, \left. \bmv \right|_{\Sigma_\ast},
\bar{\nabla} \left. \bmv \right|_{\Sigma_\ast}, 
\bar{\nabla}^2 \left. \bmv \right|_{\Sigma_\ast})$
the quantities $(g^i)_{jk}$, $\det \left((g^i)_{jk} \right)$, $N^i$, and 
$(h^i)_{jk}$ are represented as
\begin{equation}
\left\{\begin{array}{ll}
(g^i)_{jk} = (g^i_\ast)_{jk} + P^i_{jk}(z^i_1) \,, 
&g^i = \det \left( (g^i)_{jk} \right) 
= \det \left( (g^i_\ast)_{jk} \right) + P^i(z^i_1)  \,, \\[0.05cm]%
N^i = N^i_\ast \, R^i(z^i_1) + Q^i(z^i_1) \,, 
&h^i_{jk} = (h^i_\ast)_{jk} \, R^i(z^i_1) + S^i_{jk}(z^i_2) \,, 
\end{array}\right.
\label{precise_form}
\end{equation}
where $P^i_{jk}$ and $P^i$ are polynomial functions with $P^i_{jk}(0)=0$ 
and $P^i(0)=0$, and $R^i$, $Q^i$ and $S^i_{jk}$ are rational functions with 
$R^i(0)=1$, $Q^i(0)=0$ and $S^i_{jk}(0)=0$. From Remark~\ref{rem:localdiff} 
we know that $g^i \neq 0$ for $v^i$ small enough in the $C^1$-norm and that 
$\partial_1 \Phi^i,\ldots,\partial_n \Phi^i$ are linearly independent, 
in particular $|\partial_1 \Phi^i \times \ldots \times \partial_n \Phi^i| 
\neq 0$, and therefore also $N^i$ is well-defined. 

Now fix $R > 0$ and define the set 
\begin{align} \label{def:D_R}
{\cal D}_R &= \bigl\{ (v^1,v^2,v^3) \in {\cal X}_\delta \,\big|\, 
v^i(\sigma,0)=\rho^i_0 ,\, 
\sum\nolimits_{i=1}^3 
\|v^i - \rho^i_0\|_{C^{2+\alpha,1+\frac{\alpha}{2}}(Q^i_{\delta})} 
\le R \bigr\} .
\end{align}
For $\bmv \in \mathcal{D}_R$ we deduce from a standard estimate 
for parabolic H\"older spaces, see e.g.
Lunardi \cite[Lem. 5.1.1]{Lu}, that for all $t \in [0,T]$ we have 
\begin{align} 
\begin{split}
\sum_{i=1}^3 \|v^i(\cdot\,,t)\|_{C^2(\G^i_\ast)} 
& \leq \sum_{i=1}^3 \|v^i(\cdot\,,t) - \rho^i_0\|_{C^2(\G^i_\ast)} 
       + \sum_{i=1}^3 \|\rho_0^i\|_{C^2(\G^i_\ast)} \\
& \leq \left( \delta^{\frac{\alpha}{2}} + C \delta^{\frac{1+\alpha}{2}} + \delta \right) 
        \|v^i - \rho^i_0\|_{C^{2+\alpha,1+\frac{\alpha}{2}}(Q^i_{\delta})} + \varepsilon_0 \\
& \leq \tilde{C}\delta^{\frac{\alpha}2}\sum_{i=1}^3 
       \|v^i - \rho^i_0\|_{C^{2+\alpha,1+\frac{\alpha}{2}}(Q^i_{\delta})} 
       + \varepsilon_0 \\
& \leq \tilde{C}\delta^{\frac{\alpha}2} R + \varepsilon_0 \,, \label{inequ:C1}
\end{split}
\end{align}
where the positive constant $\tilde{C}$ depends only on $\alpha$
and $\max\{1,\delta^{1-\frac{\alpha}2}\}$. 
This shows that for sufficiently small $\delta$ and $\varepsilon_0$ 
the operators $\mathcal{F}^i$, $\mathfrak{a}^i_\dagger$ and $\mathfrak{b}^i$, 
evaluated at functions of the form $v^i(\cdot,t)$, satisfy (A1)-(A3) for 
all $t \in [0,\delta]$. In particular we remark for later use that for 
the right hand side $\mathcal{K}^i$ of the first line in 
\eqref{prob:nonlinearnonlocal}, which is a combination of terms of 
the form $\mathcal{F}^i$ and $\mathfrak{a}^i_\dagger$, we can conclude 
an analogue statement as in (A1)-(A2). This means that for 
$\bmv$, $\bmw \in \mathcal{D}_R$ the operator $\mathcal{K}^i$ is well-defined 
and it holds 
\begin{align} \label{eq:firstlipforK}
\| D_{\sbmv} \mathcal{K}^i(v^i,\left. \bmv \right|_{\Sigma_\ast}) 
-  D_{\sbmv} \mathcal{K}^i(w^i,\left. \bmw \right|_{\Sigma_\ast}) \|_\infty
& \leq L \sum_{i=1}^3 \| v^i - w^i \|_{C^{2+\alpha,1+\frac{\alpha}{2}}
(Q^i_{\delta})} \,,
\end{align}
where $D_{\sbmv}$ 
is any first order derivative in 
$\{\partial_{v^i}, \partial_{\nabla_k v^i}, 
\partial_{\nabla_{kj}^2 v^i}, \partial_{\left. \sbmv \right|_{\Sigma_\ast}},
\partial_{\overline{\nabla}_k \left. \sbmv \right|_{\Sigma_\ast}}, 
\partial_{\overline{\nabla}_{kj}^2 \left. \sbmv \right|_{\Sigma_\ast}} \}$. 
Note that $L$ depends only on the chosen $r>0$ from the beginning 
of the proof. In particular the same estimate holds true for 
$\bmv = \brho_0$ and $\bmw = \boldsymbol{0}$, i.e. 
\begin{align} \label{eq:secondlipforK}
\| D_{\sbmv} \mathcal{K}^i(\rho^i_0,\left. \brho_0 \right|_{\Sigma_\ast}) 
-  D_{\sbmv} \mathcal{K}^i(0) \|_\infty 
& \leq L \sum_{i=1}^3 \| \rho^i_0 \|_{C^{2+\alpha}(\G^i_\ast)} \,.
\end{align}
Due to the Lipschitz-continuity we also have that $D_{\sbmv} \mathcal{K}^i$ 
is bounded as a mapping from 
$\mathcal{D}_R \subset C^{2+\alpha, 1+\frac{\alpha}{2}}(Q^i_\delta)$ into 
$C^{\alpha,\frac{\alpha}{2}}(Q^i_\delta)$, which will be used later to 
estimate 
\begin{align} \label{eq:thirdlipforK}
\left[ D_{\sbmv} \mathcal{K}^i(v^i,\left. \bmv \right|_{\Sigma_\ast}) \right]
_{C^{\alpha,\frac{\alpha}{2}}} &\leq C(R) \,.
\end{align}

Fix $\bmv = (v^1,v^2,v^3) \in \mathcal{D}_R$ and let 
$\bmu = (u^1,u^2,u^3) = \Lambda(\bmv)$ be the solution of the linear,
nonhomogeneous problem for $i=1,2,3$:
\begin{equation} \label{Lambda_problem}
  \left\{\begin{array}{ll}
           \partial_t u^i={\cal A}^iu^i+\zeta^i({\cal T}(\bmu\circ\pr^i))^i
              +\mathfrak{f}^i(v^i,\left.\bmv\right|_{\Sigma_\ast})
              & \mbox{on } \; \G^i_\ast \times[0,\delta], \\[0.1cm]
           \dis\sum_{j=1}^3{\cal B}^{ij}u^j=\mathfrak{b}^i(\bmv) 
              & \mbox{on } \; \Sigma_\ast \times [0,\delta], \\[0.6cm]
           u^i\big|_{t=0}=\rho^i_0 & \mbox{on } \; \G^i_\ast \,.
\end{array}\right.
\end{equation}
Due to the compatibility condition \eqref{comp:rho_0} for $\brho_0$, 
we see that $\mathfrak{f}^i$ and $\mathfrak{b}^i$ satisfy the necessary 
compatibility conditions to apply Remark \ref{rem:initialdate}, 
that is
\begin{align*}
\sum_{i=1}^3 \gamma^i \left. \mathfrak{f}^i(v^i,\left.\bmv\right|_{\Sigma_\ast})\right|_{t=0} 
= -\sum_{i=1}^3 \gamma^i \left( \mathcal{A}^i \rho^i_0 
  + \zeta \left( \mathcal{T} \brho_0 \right)^i \right)
  \, \mbox{ and } \,
\left. \mathfrak{b}^i(\bmv) \right|_{t=0} 
= \sum_{j=1}^3 \mathcal{B}^{ij} \rho^j_0 \quad \mbox{ on } \; \Sigma_\ast \,.
\end{align*}
Therefore we get a unique solution $\bmu \in \mathcal{X}_\delta$ of 
\eqref{Lambda_problem} for given $\bmv \in \mathcal{D}_R$ for 
a possibly smaller $\delta > 0$, but not depending on the choice of 
$\bmv \in \mathcal{D}_R$. 

If we are now able to find a fixed point of $\Lambda$, then this is a local 
solution to the nonlinear problem \eqref{L_eq}. Thus we will prove that 
$\Lambda$ maps ${\cal D}_R$ into itself and is a contraction for suitable 
$\delta$, $\varepsilon_0$ and $R$. 

For $\bmv, \bmw \in {\cal D}_R$ we see that 
$\bmu = \Lambda(\bmv) - \Lambda(\bmw)$ is the solution of 
\begin{equation} \label{L_eq_diff}
\left\{ \begin{array}{ll} 
\partial_t u^i 
= {\cal A}^i u^i +\zeta^i({\cal T}(\bmu\circ\pr^i))^i 
  + \mathfrak{f}^i(v^i,\left.\bmv\right|_{\Sigma_\ast})
  -\mathfrak{f}^i(w^i,\left.\bmw\right|_{\Sigma_\ast})
    & \mbox{on } \; \G^i_\ast \times [0,\delta], \\[0.1cm]
\dis\sum_{j=1}^3{\cal B}^{ij}u^j 
= \mathfrak{b}^i(\bmv)-\mathfrak{b}^i(\bmw)
    & \mbox{on } \; \Sigma_\ast \times [0,\delta], \\[0.6cm]
u^i(.\,,0)=0 & \mbox{on } \; \G^i_\ast
\end{array} \right.
\end{equation}
for $i=1,2,3$. Then, by means of Theorem~\ref{thm:L_exist}, we have 
the estimate 
\begin{align*}
&\sum_{i=1}^3\|u^i\|_{C^{2+\alpha,1+\frac{\alpha}2}(Q^i_{\delta})} \\
&\le C\sum_{i=1}^3\Bigl\{ 
\|\mathfrak{f}^i(v^i,\left.\bmv\right|_{\Sigma_\ast})
  -\mathfrak{f}^i(w^i,\left.\bmw\right|_{\Sigma_\ast})\|
_{C^{\alpha,\frac{\alpha}2}(Q^i_{\delta})} 
+\|\mathfrak{b}^i(\bmv)-\mathfrak{b}^i(\bmw)\|
_{C^{1+\alpha,\frac{1+\alpha}2}(\Sigma_\ast \times [0,\delta])}\Bigr\}.
\end{align*}
Now we claim that there are constants $C(R)$ and $L$ such that 
\begin{align} 
\begin{split}
&\sum_{i=1}^3\Bigl\{ 
\|\mathfrak{f}^i(v^i,\left.\bmv\right|_{\Sigma_\ast}) 
  - \mathfrak{f}^i(w^i,\left.\bmw\right|_{\Sigma_\ast})\|
_{C^{\alpha,\frac{\alpha}2}(Q^i_{\delta})} 
+ \|\mathfrak{b}^i(\bmv)-\mathfrak{b}^i(\bmw)\|
_{C^{1+\alpha,\frac{1+\alpha}2}(\Sigma_\ast \times [0,\delta])}\Bigr\}  \\
& \le \left( C(R)\delta^{\frac{\alpha}2} 
+ L\varepsilon_0 \right) 
  \sum_{i=1}^3\|v^i-w^i\|_{C^{2+\alpha,1+\frac{\alpha}2}(Q^i_{\delta})}, \label{key_estimate}
\end{split}
\end{align}
where $C(R)$ is independent of $\delta$ and $L$ is as in 
(\ref{eq:firstlipforK}). 
To show the estimate for $\mathfrak{f}^i$, we use the notation 
$\mathcal{A}^i_{\mbox{\footnotesize all}} \bmv 
= \mathcal{A}^i \bmv + \zeta^i (\mathcal{T} (\bmv \circ \pr^i))^i$ 
for the linearization including the nonlocal terms to get, 
compare \eqref{equ:defoff},
\begin{align*}
\mathfrak{f}^i(v^i,\left.\bmv\right|_{\Sigma_\ast}) 
= \mathcal{K}^i(v^i,\left.\bmv\right|_{\Sigma_\ast})
  - \mathcal{A}^i_{\mbox{\footnotesize all}} \bmv \,.
\end{align*}
Note that herein $\mathcal{A}^i_{\mbox{\footnotesize all}} \bmv 
= \partial \mathcal{K}^i(0) \bmv$ is the linearization around 
the reference hypersurfaces represented through $\brho = 0$ and 
that $\mathcal{K}^i$ is a nonlinear nonlocal operator depending on 
$v^i$, $\nabla v^i$, $\nabla^2 v^i$, $\left.\bmv\right|_{\Sigma_\ast}$, 
$\bar{\nabla} \left.\bmv\right|_{\Sigma_\ast}$ and
$\bar{\nabla}^2 \left.\bmv\right|_{\Sigma_\ast}$, compare \eqref{equ:tildeK}.

The difference in $\mathfrak{f}^i$ can be written locally with the help of 
a suitable parametrization as follows
\begin{align*}
&\hspace*{-10pt}
\mathfrak{f}^i(v^i,\left.\bmv\right|_{\Sigma_\ast}) 
- \mathfrak{f}^i(w^i,\left.\bmw\right|_{\Sigma_\ast}) \\
=& \int_0^1 \frac{d}{d s} 
   \mathcal{K}^i(\xi_s(v^i,w^i,\left.\bmv\right|_{\Sigma_\ast},
   \left.\bmw\right|_{\Sigma_\ast}))\,ds 
   - \mathcal{A}^i_{\mbox{\footnotesize all}}(\bmv - \bmw) \\
=& \,\Theta^i(v^i,w^i,\left.\bmv\right|_{\Sigma_\ast},
     \left.\bmw\right|_{\Sigma_\ast}) (v^i-w^j)
   + \sum_{j=1}^3 \bar{\Theta}^{i,j}(v^i,w^i,
     \left.\bmv\right|_{\Sigma_\ast},\left.\bmw\right|_{\Sigma_\ast}) 
     (\left. v^j \right|_{\Sigma_\ast} - \left. w^j \right|_{\Sigma_\ast}) \\
 & + \sum_{k=1}^n \Theta^i_k(v^i,w^i,\left.\bmv\right|_{\Sigma_\ast},
     \left.\bmw\right|_{\Sigma_\ast}) \nabla_k(v^i-w^j) \\
 & + \sum_{j=1}^3 \sum_{k=1}^{n-1} \bar{\Theta}^{i,j}_k(v^i,w^i,
     \left.\bmv\right|_{\Sigma_\ast},\left.\bmw\right|_{\Sigma_\ast}) 
     \bar{\nabla}_k (\left. v^j \right|_{\Sigma_\ast} 
     - \left. w^j \right|_{\Sigma_\ast}) \\
 & + \sum_{k,l=1}^n \Theta^i_{k,l}(v^i,w^i,
     \left.\bmv\right|_{\Sigma_\ast},\left.\bmw\right|_{\Sigma_\ast}) 
     \nabla_{kl}^2(v^i-w^j) \\
 & + \sum_{j=1}^3 \sum_{k,l=1}^{n-1} \bar{\Theta}^{i,j}_{k,l}(v^i,w^i,
     \left.\bmv\right|_{\Sigma_\ast},\left.\bmw\right|_{\Sigma_\ast}) 
     \bar{\nabla}^2_{kl} (\left. v^j \right|_{\Sigma_\ast} 
     - \left. w^j \right|_{\Sigma_\ast}) \\
 & + \left( \partial \mathcal{K}^i(\rho^i_0, 
     \left. \brho_0 \right|_{\Sigma_\ast}) - \partial \mathcal{K}^i(0) \right) 
     (\bmv - \bmw) \,,
\end{align*}
where with $\xi_0 = (\rho^i_0, \left. \brho_0 \right|_{\Sigma_\ast})$ we use 
the following notation
\begin{align*}
\xi_s(v^i,w^i,\left.\bmv\right|_{\Sigma_\ast},\left.\bmw\right|_{\Sigma_\ast}) 
&= \left( s v^i + (1-s) w^i, s \left.\bmv\right|_{\Sigma_\ast} 
          + (1-s) \left.\bmw\right|_{\Sigma_\ast} \right), \\
\Theta^i(v^i,w^i,\left.\bmv\right|_{\Sigma_\ast},
         \left.\bmw\right|_{\Sigma_\ast}) 
&= \int_0^1 \left( \partial_{v^i} 
   \mathcal{K}^i(\xi_s(v^i,w^i,\left.\bmv\right|_{\Sigma_\ast},
                 \left.\bmw\right|_{\Sigma_\ast}))
   - \partial_{v^i} \mathcal{K}^i(\xi_0) \right)\,ds, \\
\bar{\Theta}^{i,j}(v^i,w^i,\left.\bmv\right|_{\Sigma_\ast},
                   \left.\bmw\right|_{\Sigma_\ast}) 
&= \int_0^1 \left( \partial_{\left.v^j\right|_{\Sigma_\ast}} 
   \mathcal{K}^i(\xi_s(v^i,w^i,\left.\bmv\right|_{\Sigma_\ast},
                 \left.\bmw\right|_{\Sigma_\ast}))
   - \partial_{\left.v^j\right|_{\Sigma_\ast}} \mathcal{K}^i(\xi_0) \right)
   \,ds,  \\
\Theta^i_k(v^i,w^i,\left.\bmv\right|_{\Sigma_\ast},
           \left.\bmw\right|_{\Sigma_\ast}) 
&= \int_0^1 \left( \partial_{\nabla_k v^i} 
   \mathcal{K}^i(\xi_s(v^i,w^i,\left.\bmv\right|_{\Sigma_\ast},
                 \left.\bmw\right|_{\Sigma_\ast}))
   - \partial_{\nabla_k v^i} \mathcal{K}^i(\xi_0) \right)\,ds, \\
\bar{\Theta}^{i,j}_k(v^i,w^i,\left.\bmv\right|_{\Sigma_\ast},
                     \left.\bmw\right|_{\Sigma_\ast}) 
&= \int_0^1 \left( \partial_{\bar{\nabla}_k \left.v^j\right|_{\Sigma_\ast}} 
   \mathcal{K}^i(\xi_s(v^i,w^i,\left.\bmv\right|_{\Sigma_\ast},
                 \left.\bmw\right|_{\Sigma_\ast}))
   - \partial_{\bar{\nabla}_k \left.v^j\right|_{\Sigma_\ast}} 
     \mathcal{K}^i(\xi_0) \right)\,ds, \\
\Theta^i_{k,l}(v^i,w^i,\left.\bmv\right|_{\Sigma_\ast},
               \left.\bmw\right|_{\Sigma_\ast}) 
&= \int_0^1 \left( \partial_{\nabla_{kl}^2 v^i} 
   \mathcal{K}^i(\xi_s(v^i,w^i,\left.\bmv\right|_{\Sigma_\ast},
                 \left.\bmw\right|_{\Sigma_\ast}))
   - \partial_{\nabla_{kl}^2 v^i} \mathcal{K}^i(\xi_0) \right)\,ds, \\
\bar{\Theta}^{i,j}_{k,l}(v^i,w^i,\left.\bmv\right|_{\Sigma_\ast},
                         \left.\bmw\right|_{\Sigma_\ast}) 
&= \int_0^1 \left(\partial_{\bar{\nabla}_{kl}^2 \left.v^j\right|_{\Sigma_\ast}}
   \mathcal{K}^i(\xi_s(v^i,w^i,\left.\bmv\right|_{\Sigma_\ast},
                 \left.\bmw\right|_{\Sigma_\ast}))
   - \partial_{\bar{\nabla}_{kl}^2 \left.v^j\right|_{\Sigma_\ast}} 
     \mathcal{K}^i(\xi_0) \right)\,ds.
\end{align*}
Herein, by a slight abuse of notation, we identify the $\mathcal{K}^i$-terms 
with its localized versions.

Now we observe for $\Theta \in \{ \Theta^i, \bar{\Theta}^i, \Theta^i_k, 
\bar{\Theta}^{i,j}_k, \Theta^i_{k,l}, \bar{\Theta}^{i,j}_{k,l} \}$ 
that $\left. \Theta \right|_{t=0} = 0$, and therefore we derive
\begin{align*}
\| \Theta \|_\infty 
\leq \delta^{\frac{\alpha}2} \la \Theta \ra^{\frac{\alpha}2}_t 
\leq C(R) \, \delta^{\frac{\alpha}2} \,.
\end{align*}
Additionally (\ref{eq:secondlipforK}) gives
$$
\| (\partial  \mathcal{K}^i(\rho^i_0, \left. \brho_0 \right|_{\Sigma_\ast}) 
- \partial \mathcal{K}^i(0))(\bmv-\bmw) \|_\infty 
\leq L\sum_{i=1}^3 \|\rho_0^i \|_{C^2} 
\sum_{i=1}^3 \|v^i - w^i\|_{C^{2+\alpha,1+\frac{\alpha}{2}}(Q^i_\delta)}\,,
$$
so that we arrive at
\begin{align*}
\| \mathfrak{f}^i(v^i,\left.\bmv\right|_{\Sigma_\ast}) 
- \mathfrak{f}^i(w^i,\left.\bmw\right|_{\Sigma_\ast}) \|_\infty
\leq \bigl( C(R) \delta^{\frac{\alpha}2} + L\varepsilon_0 \bigr)
  \sum_{i=1}^3 \|v^i - w^i\|_{C^{2+\alpha,1+\frac{\alpha}{2}}(Q^i_\delta)} \,.
\end{align*}
Moreover it follows from $\left. D(v^i-w^i) \right|_{t=0} = 0$, where 
$D \in \{\nabla^0, \nabla_k, \nabla^2_{k,l},\bar{\nabla}^0, \bar{\nabla}_{k}, 
\bar{\nabla}^2_{k,l} \}$ (of course for surface gradients $\bar{\nabla}$ 
we restrict the function $v^i-w^i$ to the triple junction $\Sigma_\ast$), that
\begin{align*}
\| D(v^i-w^i) \|_\infty 
\leq \delta^{\frac{\alpha}2} \la D(v^i-w^i) \ra_t^{\frac{\alpha}2}
\leq \delta^{\frac{\alpha}2} 
     \sum_{i=1}^3\|v^i - w^i\|_{C^{2+\alpha,1+\frac{\alpha}{2}}(Q^i_\delta)}\,.
\end{align*}
Set $[\,\cdot\,]_{C^{\alpha,\frac{\alpha}2}} = \la\,\cdot\,\ra_x^{\alpha} 
+ \la\,\cdot\,\ra_t^{\frac{\alpha}2}$ and let $\Theta D(v^i-w^i)$ be 
corresponding to each other as in the formula for the difference in 
$\mathfrak{f}^i$. Then we obtain
\begin{align*}
\left[ \Theta D(v^i-w^i) \right]_{C^{\alpha,\frac{\alpha}2}} 
& \leq \|\Theta\|_\infty \left[D(v^i-w^i)\right]_{C^{\alpha,\frac{\alpha}2}} 
  + \left[\Theta\right]_{C^{\alpha,\frac{\alpha}2}} \|D(v^i-w^i)\|_\infty \\
& \leq C(R) \delta^{\frac{\alpha}2} 
  \sum_{i=1}^3 \|v^i - w^i\|_{C^{2+\alpha,1+\frac{\alpha}{2}}(Q^i_\delta)} \,.
\end{align*}
Additionally it follows from (\ref{eq:secondlipforK}) and 
(\ref{eq:thirdlipforK}) that
$$
\left[ \left( 
\partial \mathcal{K}^i(\rho^i_0, \left. \brho_0 \right|_{\Sigma_\ast}) 
- \partial \mathcal{K}^i(0) \right) (\bmv - \bmw) \right]
_{C^{\alpha,\frac{\alpha}2}}
\leq (C(R)\delta^{\frac{\alpha}2}+L\varepsilon_0)
  \sum_{i=1}^3 \|v^i - w^i\|_{C^{2+\alpha,1+\frac{\alpha}{2}}(Q^i_\delta)}\,.
$$
Thus we are led to
\begin{align*}
\left[ \mathfrak{f}^i(v^i,\left.\bmv\right|_{\Sigma_\ast}) 
- \mathfrak{f}^i(w^i,\left.\bmw\right|_{\Sigma_\ast}) \right]
_{C^{\alpha,\frac{\alpha}2}}
\le \left( C(R)\delta^{\frac{\alpha}2} + L \varepsilon_0 \right)
\sum_{i=1}^3\|v^i-w^i\|_{C^{2+\alpha,1+\frac{\alpha}{2}}(Q^i_{\delta})} \, .
\end{align*}
By using (A3) we can give analogously an estimate for the differences 
in $\mathfrak{b}^i$ and therefore we arrive at the inequality 
(\ref{key_estimate}). 
Consequently, we obtain that $\Lambda$ is a $1/2$-contraction provided 
$\delta$ and $\varepsilon_0$ are small enough.

To see that $\Lambda$ maps $\mathcal{D}_R$ into itself, we have for 
$\bmv \in \mathcal{D}_R$ and $\bmu = \Lambda(\bmv)$
\begin{align*}
\sum_{i=1}^3 
\| u^i - \rho_0^i \|_{C^{2+\alpha,1+\frac{\alpha}{2}}(Q^i_{\delta})} 
& \leq \sum_{i=1}^3 \left(
\|\Lambda(\bmv)^i-\Lambda(\brho_0)^i\|
_{C^{2+\alpha,1+\frac{\alpha}{2}}(Q^i_{\delta})}
+\|\Lambda(\brho_0)^i-\rho_0^i\|
_{C^{2+\alpha,1+\frac{\alpha}{2}}(Q^i_{\delta})}\right) \\
& \leq \frac{R}{2} 
  + \sum_{i=1}^3 \| \Lambda(\brho_0)^i - \rho_0^i \|
                 _{C^{2+\alpha,1+\frac{\alpha}{2}}(Q^i_{\delta})} .
\end{align*}
For the second inequality, we used the fact that $\Lambda$ is 
a $1/2$-contraction provided $\delta$ and $\varepsilon_0$ are small enough.
The function $\bmw = \Lambda(\brho_0) - \brho_0$ is the solution of
\begin{equation} \label{L_eq_diffrho_0}
\left\{ \begin{array}{ll} 
  \partial_t w^i 
  = {\cal A}_{\mbox{\footnotesize all}}^i w^i 
    + \mathcal{K}^i(\rho_0^i,\left.\brho_0\right|_{\Sigma_\ast})
       & \mbox{on } \; \G^i_\ast \times [0,\delta], \\[0.1cm]
  \dis\sum_{j=1}^3{\cal B}^{ij}w^j = 0
       & \mbox{on } \; \Sigma_\ast \times [0,\delta], \\[0.6cm]
  w^i(\cdot\,,0)=0 & \mbox{on } \; \G^i_\ast \,.
\end{array}\right.
\end{equation}
Due to the assumptions \eqref{comp:rho_0} on $\brho_0$ the compatibility 
conditions from Theorem \ref{thm:L_exist} are fulfilled and we can apply it 
to get the existence of a $C>0$ independent of $\delta > 0$, such that 
the solution $\bmw$ of \eqref{L_eq_diffrho_0} satisfies
\begin{align*}
\sum_{i=1}^3 \| w^i \|_{C^{2+\alpha,1+\frac{\alpha}{2}}(Q^i_{\delta})} 
& \leq C \sum_{i=1}^3  
  \| \mathcal{K}^i(\rho_0^i,\left.\brho_0\right|_{\Sigma_\ast}) \|
  _{C^{\alpha,\frac{\alpha}2}} .
\end{align*}
We estimate the right side of the above inequality by $C'=C'(\varepsilon_0)$ 
and we arrive at
\begin{align*}
\sum_{i=1}^3 
\| u^i - \rho_0^i \|_{C^{2+\alpha,1+\frac{\alpha}{2}}(Q^i_{\delta})}
& \leq \frac{R}{2} + C' \,.
\end{align*}
Therefore for $R$ suitably large enough $\Lambda$ maps $\mathcal{D}_R$ 
into itself.
In the following we illustrate the choice of the constants 
in detail. 
First we choose $\varepsilon_0 > 0$ such that 
$L \varepsilon_0 < 1/4$ and $\varepsilon_0 < r/2$. 
Then we choose $R_0>0$ such that $C'(\varepsilon_0) < R_0/2$,
which means that $R_0/2 + C'(\varepsilon_0) < R_0$. Now for 
a given arbitrary but fixed $R \geq R_0$ we choose $\delta > 0$ 
such that
\begin{align*}
\tilde{C}\delta^{\frac{\alpha}2}R < \frac{r}{2} \quad \mbox{ and } \quad 
C(R) \delta^{\frac{\alpha}{2}} < \frac{1}{4} \,,
\end{align*}
where the constants  $\tilde{C}$, $C(R)$ are from inequalities 
\eqref{inequ:C1} and \eqref{key_estimate}.
With this choice of $\varepsilon_0$, $R$ and $\delta$ we observe
$$
\begin{array}{ll}
\tilde{C}\delta^{\frac{\alpha}2}R + \varepsilon_0 < r 
&\quad \mbox{(such that the properties (A1)-(A3) are fulfilled)}, \\[0.05cm]
\dfrac{R}{2} + C'(\varepsilon_0) < R 
&\quad \mbox{(such that $\Lambda$ is a self mapping)}, \\[0.05cm]
C(R) \delta^{\frac{\alpha}{2}} + L \varepsilon_0 < \dfrac{1}{2} 
&\quad \mbox{(such that $\Lambda$ is a $1/2$-contraction)}.
\end{array}
$$
Therefore we conclude that $\Lambda$ has a unique fixed point in 
$\mathcal{D}_R$, which was the remaining part to prove the theorem. 
\end{proof}

\begin{remark}[A continuation criteria] \label{rem:timeextension}
The question arises on which interval $[0,T_{\mbox{\footnotesize max}})$ the
mean curvature flow with triple junction \eqref{MCF1}, \eqref{MCF2} can 
be extended. A careful revision of the above proof shows that $\delta$ in the
local existence interval
depends on the size of~$r$ (responsible for the validity of 
Assumptions (A1)-(A3)) and on $\varepsilon_0$. We note that for the validity 
of Assumptions (A1)-(A3) we need that the metric tensor is positive definite 
and in particular that the inverse exists. In Remark \ref{rem:localdiff}
we gave a formula for the metric tensor and one can see that if 
the second fundamental form of $\G^i_\ast$ and terms $\partial_l \tau^i_\ast$ 
are bounded, we can give a lower bound on the choice of $r$. 
If in addition we choose $\varepsilon_0$ small enough, this would lead to 
a lower bound on the existence interval $[0,\delta]$. In this way, 
we can achieve existence in any given time interval $[0,T]$ by splitting it 
into small ones and by choosing appropriate reference configurations on 
each interval, providing the $\partial_l \tau^i_\ast$ can be chosen bounded
for all reference configurations on the interval $[0,T]$. 

 We remark that the bound on $\partial_l \tau^i_\ast$ can be achieved 
in the following way. If we choose the vector $\tau^i_\ast$ as a truncation 
of the unit outer conormal with the help of geodesic lines,
we can do this in a strip around $\partial \G^i_\ast$ given by 
$q + r \nu^i_\ast(q)$, where $q \in \partial \G^i_\ast$ and
$0 \leq r \leq r_0$ for some positive $r_0$. Here we replace $r$ by 
a cut-off function evaluated at the geodesic distance from 
$\partial \G^i_\ast$. This gives a minimal bound on the diameter of 
the neighbourhood of the triple junction, where $\tau^i_\ast$ does not 
vanish and in this way we can also bound derivatives of the form 
$\partial_l \tau^i_\ast$. Possible scenarios for which this cannot be 
achieved are the following:
\begin{itemize}
 \item The area of one hypersurface converges to zero.
 \item The triple junction develops during the evolution a self contact. 
\end{itemize}

A similar continuation criterion in the case of curves has been studied
in Mantegazza, Novaga and 
Tortorelli~\cite{MNT}, where the authors consider evolution of planar 
networks according to curvature flow and conclude existence as long as one 
of the length of the curves tends to zero or a curvature integral blows up 
at a certain minimal rate. 
\end{remark}

\begin{remark}[Cluster with boundary contact] \label{rem:fixedboundary}
 We remark that it is also possible to consider a configuration where 
the three hypersurfaces lie inside a fixed bounded region 
$\Omega \subset \mathbb{R}^{n+1}$ and meet its boundary at 
a given contact angle, see for example Bronsard and Reitich~\cite{BR} or 
Garcke, Kohsaka and \v{S}ev\v{c}ovi\v{c}~\cite{GKS} for curves in the plane, 
and Depner~\cite{Dep10} or Depner and Garcke~\cite{DG11} for 
arbitrary dimensions. A natural contact angle achieved by the minimization 
of the weighted area would be 90 degree. If one uses the parametrization 
of \cite{Dep10} or \cite{DG11} to describe the geometric problem as 
a system of partial differential equations and the ideas from \cite{BR}, 
\cite{GKS} or from this work, one could derive a local existence result 
also in this situation. 
\end{remark}



\end{document}